\documentclass[a4paper,12pt,reqno]{amsart}
\usepackage{amsfonts}
\usepackage{amsmath}
\usepackage{amssymb}
\usepackage[a4paper]{geometry}
\usepackage{mathrsfs}
\usepackage{csquotes}

\usepackage[colorlinks]{hyperref}
\renewcommand\eqref[1]{(\ref{#1})} 
\numberwithin{equation}{section}
\theoremstyle{plain}
\newtheorem{thm}{Theorem}[section]

\theoremstyle{definition}

\newtheorem{rem}[thm]{Remark}

\newcommand{\Rn}{\mathbb R^{n}}

\def\e[#1]{{\textrm{e}}^{#1}}

\begin{document}

   \title[Hardy, TM and CKN inequalities]
 {Hardy, weighted Trudinger-Moser and Caffarelli-Kohn-Nirenberg type inequalities on Riemannian manifolds with negative curvature}

\author[M. Ruzhansky]{Michael Ruzhansky}
\address{
  Michael Ruzhansky:
  \endgraf
  Department of Mathematics
  \endgraf
  Imperial College London
  \endgraf
  180 Queen's Gate, London SW7 2AZ
  \endgraf
  United Kingdom
  \endgraf
  {\it E-mail address} {\rm m.ruzhansky@imperial.ac.uk}
  }

\author[N. Yessirkegenov]{Nurgissa Yessirkegenov}
\address{
  Nurgissa Yessirkegenov:
  \endgraf
  Institute of Mathematics and Mathematical Modelling
  \endgraf
  125 Pushkin str.
  \endgraf
  050010 Almaty
  \endgraf
  Kazakhstan
  \endgraf
  and
  \endgraf
  Department of Mathematics
  \endgraf
  Imperial College London
  \endgraf
  180 Queen's Gate, London SW7 2AZ
  \endgraf
  United Kingdom
  \endgraf
  {\it E-mail address} {\rm n.yessirkegenov15@imperial.ac.uk}
  }

\thanks{The authors were supported in parts by the EPSRC
 grant EP/R003025/1 and by the Leverhulme Grant RPG-2017-151, as well as by the MESRK grant AP05133271. No new data was collected or
generated during the course of research.}


     \begin{abstract} In this paper we obtain Hardy, weighted Trudinger-Moser and Caffarel\-li-Kohn-Nirenberg type inequalities with sharp constants on Riemannian manifolds with non-positive sectional curvature and, in particular, a variety of new estimates on hyperbolic spaces. Moreover, in some cases we also show their equivalence with Trudinger-Moser inequalities. As consequences, the relations between the constants of these inequalities are investigated yielding asymptotically best constants in the obtained inequalities. We also obtain the corresponding uncertainty type principles.
     \end{abstract}
     \maketitle

       \tableofcontents

\section{Introduction}
\label{SEC:intro}

Recall the classical Caffarelli-Kohn-Nirenberg inequality \cite{CKN84}:

\begin{thm}\label{clas_CKN}
Let $n\in\mathbb{N}$ and let $p_{1}$, $p_{2}$, $p_{3}$, $a$, $b$, $d$, $\delta\in \mathbb{R}$ be such that $p_{1},p_{2}\geq1$,
$p_{3}>0$, $0\leq\delta\leq1$, and
\begin{equation}\label{clas_CKN0}
\frac{1}{p_{1}}+\frac{a}{n},\, \frac{1}{p_{2}}+\frac{b}{n},\, \frac{1}{p_{3}}+\frac{c}{n}>0,
\end{equation}
where $c=\delta d + (1-\delta) b$. Then there exists a positive constant $C$ such that
\begin{equation}\label{clas_CKN1}
\||x|^{c}f\|_{L^{p_{3}}(\Rn)}\leq C \||x|^{a}|\nabla f|\|^{\delta}_{L^{p_{1}}(\Rn)} \||x|^{b}f\|^{1-\delta}_{L^{p_{2}}(\Rn)}
\end{equation}
holds for all $f\in C_{0}^{\infty}(\Rn)$, if and only if the following conditions hold:
\begin{equation}\label{clas_CKN2}
\frac{1}{p_{3}}+\frac{c}{n}=\delta
\left(\frac{1}{p_{1}}+\frac{a-1}{n}\right)+(1-\delta)\left(\frac{1}{p_{2}}+\frac{b}{n}\right),
\end{equation}
\begin{equation}\label{clas_CKN3}
a-d\geq 0 \quad \textrm{if} \quad \delta>0,
\end{equation}
\begin{equation}\label{clas_CKN4}
a-d\leq 1 \quad \textrm{if} \quad \delta>0 \quad \textrm{and} \quad \frac{1}{p_{3}}+\frac{c}{n}=\frac{1}{p_{1}}+\frac{a-1}{n}.
\end{equation}
\end{thm}

It is a natural problem, also important for applications, to find an analogue of the above Caffarelli-Kohn-Nirenberg inequalities on Lie groups or on Riemannian
manifolds. On Lie groups, we refer, for example, to \cite{ZHD14} for Heisenberg groups, to \cite{Yac18} for Lie groups of polynomial
volume growth, to \cite{RSY17_strat} and to \cite{RS17_strat} for stratified groups, to
\cite{RSY17_hom1}, to \cite{RSY17_hom2} and to \cite{ORS17_hom} for general homogeneous groups. On Riemannian manifolds, in \cite{CX04}
and \cite{Mao15} the authors assuming that Caffarelli-Kohn-Nirenberg type inequalities hold, investigated the geometric property related to the volume of a geodesic ball on an $n$-dimensional $(n\geq3)$ complete open
manifold with non-negative Ricci curvature and on an $n$-dimensional $(n\geq3)$ complete and noncompact smooth metric measure space with
non-negative weighted Ricci curvature, respectively.

Recently, the following Caffarelli-Kohn-Nirenberg type inequalities have been obtained on {\em Cartan-Hadamard manifolds} $M$, that is, complete simply connected manifolds of non-positive sectional curvature, in \cite{Ngu17}:
Let $n\geq2$, $p>1$, $r>0$, $\alpha,\beta,\gamma\in\mathbb{R}$ and $\gamma=(1+\alpha)/r+(p-1)\beta/(pr)$ be such that
\begin{equation}\label{CKN_Nguyen1}
\frac{1}{r}-\frac{\gamma}{n}>0,\;\frac{1}{p}-\frac{\alpha}{n}>0,\;1-\frac{\beta}{n}>0.
\end{equation}
Then we have for all $f\in C_{0}^{\infty}(M)$
\begin{equation}\label{CKN_Nguyen2}
\int_{M}\frac{|f(x)|^{r}}{(\rho(x))^{\gamma r}}dx\leq \frac{r}{n-\gamma r}
\left(\int_{M}\frac{|\partial_{\rho}f(x)|^{p}}{(\rho(x))^{\alpha p}}dx\right)^{\frac{1}{p}}
\left(\int_{M}\frac{|f(x)|^{\frac{p(r-1)}{p-1}}}{(\rho(x))^{\beta}}dx\right)^{\frac{p-1}{p}},
\end{equation}
for $r\neq 1$, where $\partial_{\rho}$ is the radial derivation along geodesic curves, and $\rho(x)={\rm dist}(x,x_{0})$ is the geodesic distance to any fixed point $x_{0}$ of $M$.

Now, on the hyperbolic space $\mathbb{H}^{n}$ with $n\geq2$, let us recall the following another recent result on Caffarelli-Kohn-Nirenberg type inequalities for radially symmetric functions \cite{ST17}: Let $2\leq p\leq \infty$, then there exists a positive constant $c_{r}=c_{r}(n,p)$ such that for all $f\in
W^{1,2}_{0,rad}(\mathbb{H}^{n})$ we have
\begin{equation}\label{ST_CKN1}
\int_{\mathbb{H}^{n}}|\nabla_{g}f|^{2}dV_{g}\geq c_{r}(n,p)\left(\int_{\mathbb{H}^{n}}A_{p}|f|^{p}dV_{g}\right)^{\frac{2}{p}},
\end{equation}
where
$$A_{p}(r)=\frac{(f(r))^{2}(1-r^{2})^{2}}{(G(r))^{\frac{p+2}{2}}},\;1\leq p<\infty,\;A_{\infty}(r)=\frac{1}{\sqrt{G(r)}},$$
and $dV_{g}$ is the volume form (see Section \ref{SEC:prelim}), and $G(r)=\int_{r}^{1}\frac{(1-t^{2})^{n-2}}{t^{n-1}}dt$ (note that $G/(n\omega_{n-1})$ is the fundamental solution of the hyperbolic
Laplacian). Here, $W^{1,2}_{0,rad}(\mathbb{H}^{n})$ is the subspace of radially symmetric functions of $W^{1,2}_{0}(\mathbb{H}^{n})$.
Moreover, $c_{r}(n,2)=1/16$ and $c_{r}(n,\infty)=2^{n-2}\omega_{n-1}$, and $c_{r}(n,p)\leq c_{r}(n,2)(c_{r}(n,\infty))^{p-2}$ with
$2<p<\infty$.

In this paper, we introduce a class of new Caffareli-Kohn-Nirenberg type inequalities with sharp constants on hyperbolic spaces. We also do not assume that any of the functions are radially symmetric. Moreover, our method allows us to show their equivalence with Trudinger-Moser inequalities. Using this method, we actually prove Hardy type inequalities on complete, simply connected Riemannian manifold $M$ with negative curvature, and on hyperbolic space $\mathbb{H}^{n}$ for $n\geq2$ with sharp constants. Furthermore, we show extended versions of weighted Trudinger-Moser inequalities on $\mathbb{H}^{n}$ for $n\geq2$ with sharp constants. We refer to Section \ref{SEC:prelim} for precise definitions. Now, let us briefly state our main results: Let $\omega_{n-1}$ be the area of the surface of the unit $n$-ball. Then we have
\begin{itemize}
\item {\bf (Hardy type inequalities on $M$)} Let $M$ be a complete, simply connected Riemannian manifold of dimension $n\geq2$ with negative curvature. Let $0\leq\beta<n$. Then for any $n\leq q<\infty$ there exists a positive constant $C_{1}=C_{1}(n, \beta, q, M)$ such that
\begin{equation}\label{Hardy_manif1_intro}
\left\|\frac{f}{\rho^{\frac{\beta}{q}}}\right\|_{L^{q}(M)}\leq C_{1}q^{1-1/n}\|f\|_{W^{1,n}(M)}
\end{equation}
holds for all functions $f\in W_{0}^{1,n}(M)$, and such that $\underset{q\rightarrow\infty}{\rm limsup\;}C_{1}(n, \beta, q, M)<\infty$. The asymptotically sharp constant  for \eqref{Hardy_manif1_intro} is given in Remark \ref{rem_B1}. Here, $\rho(x)={\rm dist}(x,x_{0})$ is the geodesic distance to any fixed point $x_{0}$ of $M$. Moreover, the Hardy type inequalities \eqref{Hardy_manif1_intro} with relation \eqref{equiv_identity_Hardy} are equivalent to the weighted Trudinger-Moser inequalities \eqref{Trudinger_manif1} with $0<\alpha<\alpha_{\beta}$.
\item {\bf (Hardy type inequalities on $\mathbb{H}^{n}\;(n\geq2)$)} Let $0\leq\beta<n$. Then for any $n\leq q<\infty$ there exists a positive constant $C_{2}=C_{2}(n, \beta, q)$ such that
\begin{equation}\label{Hardy_hyper1_intro}
\left\|\frac{f}{\rho^{\frac{\beta}{q}}}\right\|_{L^{q}(\mathbb{H}^{n})}\leq C_{2}q^{1-1/n}\|\nabla_{g} f\|_{L^{n}(\mathbb{H}^{n})}
\end{equation}
holds for all functions $f\in W_{0}^{1,n}(\mathbb{H}^{n})$, and such that $\underset{q\rightarrow\infty}{\rm limsup\;}C_{2}(n, \beta, q)<\infty$. The asymptotically sharp constant for \eqref{Hardy_hyper1_intro} in the sense of Remark \ref{rem_B2} is given in Theorem \ref{Hardy_hyper_thm}. Furthermore, the Hardy type
inequalities \eqref{Hardy_hyper1_intro} with relation \eqref{equiv_identity_Hardy_hyper} are equivalent to the weighted Trudinger-Moser inequalities \eqref{Trudinger_hyper1} with $0<\alpha<\alpha_{\beta}$.
\item {\bf (Uncertainty type principle on $\mathbb{H}^{n}\;(n\geq2)$)} Let $0\leq\beta<n$. Then we have
\begin{equation}\label{uncer_1_intro}
\begin{split}
\left(\int_{\mathbb{H}^{n}}|\nabla_{g}
f(x)|^{n}dV_{g}\right)^{1/n}&\left(\int_{\mathbb{H}^{n}}\rho^{q'}|f(x)|^{q'}dV_{g}\right)^{1/q'}\\ &\geq
C_{2}^{-1}q^{1/n-1}\int_{\mathbb{H}^{n}}\rho^{\frac{q-\beta}{q}}|f(x)|^{2}dV_{g}
\end{split}
\end{equation}
holds for all functions $f\in W_{0}^{1,n}(\mathbb{H}^{n})$ and any $n\leq q<\infty$, where $1/q+1/q'=1$, and $C_{2}$ is the constant from
\eqref{Hardy_hyper1_intro}.
\item {\bf (Weighted Trudinger-Moser inequalities on $\mathbb{H}^{n}\;(n\geq2)$ I)} Let $0\leq\beta<n$ and let $0<\alpha<\alpha_{\beta}$ with $\alpha_{\beta}=n\omega_{n-1}^{1/(n-1)}(1-\beta/n)$. Then there exists a positive constant $\widetilde{C_{5}}=\widetilde{C_{5}}(\beta, n, \alpha)$ such that
\begin{equation}\label{Trudinger_GN2_hyper1_analog_intro}
\begin{split}
\int_{\mathbb{H}^{n}}&\frac{1}{(1+|f(x)|)^{n/(n-1)}\rho^{\beta}}\left(\exp(\alpha|f(x)|^{n/(n-1)})\right.\\&\left.-\sum_{k=0}^{n-2}\frac{
\alpha^{k}|f(x)|^{kn/(n-1)}}{k!}\right)dV_{g}\leq \widetilde{C_{5}}\int_{\mathbb{H}^{n}}
\frac{|f(x)|^{n}}{\rho^{\beta}}dV_{g}
\end{split}
\end{equation}
holds for all functions $f\in
W_{0}^{1,n}(\mathbb{H}^{n})$ with $\|\nabla_{g} f\|_{L^{n}(\mathbb{H}^{n})}\leq1$. Moreover, the power $n/(n-1)$ in the denominator is sharp.
\item {\bf (Weighted Trudinger-Moser inequalities on $\mathbb{H}^{n}\;(n\geq2)$ II)}
Let $0\leq \beta_{1}<n$ and $\beta_{2}\in \mathbb{R}$. Let $0<\alpha<\alpha_{\beta_{1}}$ with $\alpha_{\beta_{1}}=n\omega_{n-1}^{1/(n-1)}(1-\beta_{1}/n)$. Let $\delta$ be as in \eqref{delta}. Then there exists a positive constant
$\widetilde{C_{3}}=\widetilde{C_{3}}(n, \alpha, \beta_{1}, \beta_{2},\delta)$ such that
\begin{equation}\label{analog_LT13_1_intro}
\begin{split}
\int_{\mathbb{H}^{n}}\frac{1}{\rho^{\beta_{1}}}\left(\exp(\alpha|f(x)|^{n/(n-1)})-\sum_{k=0}^{n-2}\right.&\left.\frac{
\alpha^{k}|f(x)|^{kn/(n-1)}}{k!}\right)dV_{g} \\&\leq \widetilde{C_{3}}\left(\int_{\mathbb{H}^{n}}
\frac{|f(x)|^{n}}{\rho^{\beta_{2}}}dV_{g}\right)^{1-\delta}
\end{split}
\end{equation}
holds for all functions $f\in W_{0}^{1,n}(\mathbb{H}^{n})$ with $\|\nabla_{g}
f\|_{L^{n}(\mathbb{H}^{n})}\leq1$. Moreover, the constant $\alpha_{\beta_{1}}$ is sharp.
\item {\bf (Caffareli-Kohn-Nirenberg inequalities on $\mathbb{H}^{n}\;(n\geq2)$ I) } Let $b$, $c\in\mathbb{R}$, $0<p_{3}<\infty$ and $1<p_{2}<\infty$. Let $\delta\in(0,1]\cap\left(\frac{p_{3}-p_{2}}{p_{3}},1\right]$. Let $0\leq b(1-\delta)-c<n(1/p_{3}-(1-\delta)/p_{2})$ and $n\leq\frac{\delta
p_{2}p_{3}}{p_{2}-(1-\delta)p_{3}}$. Then we have
\begin{equation}\label{CKN_1_intro}
\|\rho^{c}f\|_{L^{p_{3}}(\mathbb{H}^{n})}
\leq \widehat{C_{3}}\|\nabla_{g} f\|^{\delta}_{L^{n}(\mathbb{H}^{n})}
\|\rho^{b}f\|^{1-\delta}_{L^{p_{2}}(\mathbb{H}^{n})}
\end{equation}
for all functions $f\in W_{0}^{1,n}(\mathbb{H}^{n})$, where $\widehat{C_{3}}=\widehat{C_{3}}(p_{2},p_{3},b,c,n,\delta)$ is given in Theorem \ref{CKN_thm}.
\item {\bf (Caffareli-Kohn-Nirenberg inequalities on $\mathbb{H}^{n}\;(n\geq2)$ II) } Let $0\leq\beta_{1}<n$ and $\beta_{2}\in\mathbb{R}$. Let $\delta$ be as in \eqref{delta}. Then for any $n\leq q<\infty$ there exists a positive constant $C_{3}=C_{3}(n, \beta_{1}, \beta_{2}, q, \delta)$ such that
\begin{equation}\label{GN_hyper1_intro}
\left\|\frac{f}{\rho^{\frac{\beta_{1}}{q}}}\right\|_{L^{q}(\mathbb{H}^{n})}\leq C_{3}q^{1-1/n}\|\nabla_{g}
f\|_{L^{n}(\mathbb{H}^{n})}^{1-\frac{n(1-\delta)}{q}}
\left\|\frac{f}{\rho^{\frac{\beta_{2}}{n}}}\right\|_{L^{n}(\mathbb{H}^{n})}^{\frac{n(1-\delta)}{q}}
\end{equation}
holds for all functions $f\in W_{0}^{1,n}(\mathbb{H}^{n})$, and such that $\underset{q\rightarrow\infty}{\rm limsup\;}C_{3}(n, \beta_{1}, \beta_{2}, q, \delta)<\infty$. The asymptotically sharp constant for \eqref{GN_hyper1_intro} in the sense of Remark \ref{rem_B3} is given in Theorem \ref{GN_hyper_thm}. Furthermore, the Caffarelli-Kohn-Nirenberg type inequalities \eqref{GN_hyper1_intro} with relation \eqref{equiv_identity_GN_hyper} are equivalent to the weighted Trudinger-Moser inequalities \eqref{analog_LT13_1_intro}.
\item {\bf (Caffareli-Kohn-Nirenberg inequalities on $\mathbb{H}^{n}\;(n\geq2)$ III) } Let $0\leq\beta<n$. Then for any $n\leq q<\infty$ there exists a positive constant $C_{5}=C_{5}(n, \beta, q)$ such that
\begin{equation}\label{GN2_hyper1_intro}
\left\|\frac{f}{\rho^{\frac{\beta}{q}}(1+|f|)^{\frac{n'}{q}}
}\right\|_{L^{q}(\mathbb{H}^{n})}\leq C_{5}q^{1-1/n} \|\nabla_{g} f\|_{L^{n}(\mathbb{H}^{n})}^{1-n/q}
\left\|\frac{f}{\rho^{\frac{\beta}{n}}}\right\|_{L^{n}(\mathbb{H}^{n})}^{n/q}
\end{equation}
holds for all functions $f\in W_{0}^{1,n}(\mathbb{H}^{n})$, and such that $\underset{q\rightarrow\infty}{\rm limsup\;}C_{5}(n, \beta, q)<\infty$. The asymptotically sharp constant for \eqref{GN2_hyper1_intro} in the sense of Remark \ref{rem_LT16_2} is given in Theorem \ref{GN3_hyper_thm}. Moreover, the Caffarelli-Kohn-Nirenberg type inequalities \eqref{GN2_hyper1_intro} with relation \eqref{equiv_identity_GN3_hyper} are equivalent to the weighted Trudin\-ger-Moser inequalities \eqref{Trudinger_GN2_hyper1_analog_intro}.

\end{itemize}

We note that the obtained Caffarelli-Kohn-Nirenberg type inequalities are not covered by \eqref{CKN_Nguyen2} and \eqref{ST_CKN1}. For example, the obtained inequality \eqref{GN_hyper1_intro}, after the change of variables $1-n/q=n/t$ for $q>n$, has the following form \begin{multline}\label{GN_hyper1_intro22}
\int_{\mathbb{H}^{n}}\frac{|f(x)|^{\frac{tn}{t-n}}}{\rho^{\beta_{1}}}dV_{g}\\ \leq C_{3}^{\frac{tn}{t-n}}\left(\frac{tn}{t-n}\right)^{\frac{t(n-1)}{t-n}}
\left(\int_{\mathbb{H}^{n}}|\nabla_{g}f(x)|^{n}dV_{g}\right)^{\frac{t}{t-n}-(1-\delta)}
\left(\int_{\mathbb{H}^{n}}\frac{|f(x)|^{n}}{\rho^{\beta_{2}}}dV_{g}\right)^{1-\delta},
\end{multline}
and holds for all functions $f\in W_{0}^{1,n}(\mathbb{H}^{n})$. Moreover, the constant $B_{3}^{\frac{tn}{t-n}}$ is asymptotically sharp for \eqref{GN_hyper1_intro22} in the sense of Remark \ref{rem_B3}, where $B_{3}$ is given in Theorem \ref{GN_hyper_thm}. Here, we see that \eqref{GN_hyper1_intro22} is not covered by \eqref{CKN_Nguyen2}, actually being completely different from \eqref{CKN_Nguyen2} in terms of parameters. We also note that \eqref{CKN_1_intro} gives different inequalities than \eqref{CKN_Nguyen2}. Indeed, for example when $p_{2}=p_{3}=n$ if we take $1+b\leq 0$ or $1+c\leq 0$ in \eqref{CKN_1_intro}, then the condition \eqref{CKN_Nguyen1} fails:
$$\frac{1}{r}-\frac{\gamma}{n}=\frac{1}{p_{3}}+\frac{p_{3}c}{rn}=\frac{1+c}{n}\leq0\;\;\text{or}\;\;1-\frac{\beta}{n}=1+\frac{bp_{2}}{n}
=1+b\leq0.$$

We also note that the obtained weighted Trudinger-Moser inequalities \eqref{Trudinger_GN2_hyper1_analog_intro} and \eqref{analog_LT13_1_intro} generalise the known results in \cite[Theorem 1]{LC17} and \cite[Theorem 1.3]{LT13}, respectively.

Of course there exist a variety of different functional and other inequalities on hyperbolic spaces. For example we can refer to \cite{RS16_BMS} for some spectral and isoperimetric inequalities for different classes of integral operators on $\mathbb{H}^{n}$, as well as to other works referred to in this paper.

This paper is organised as follows. In Section \ref{SEC:prelim} we briefly recall the main concepts of Riemannian manifolds with negative curvature and hyperbolic spaces. The Hardy type inequalities with sharp constants on $M$ and $\mathbb{H}^{n}$ are discussed in Section \ref{SEC:Hardy_manif} and in Section \ref{SEC:Hardy_hyper}, respectively. In Section \ref{SEC:CKN} we introduce Caffarelli-Kohn-Nirenberg type and weighted Trudinger-Moser inequalities with sharp constants.

\section{Preliminaries}
\label{SEC:prelim} In this section we briefly review some main concepts of Riemannian manifolds with negative curvature and refer to
\cite{GHL04}, \cite{Li93} and \cite{SY94} for more detailed information.

Let $M$ be an $n$-dimensional complete Riemannian manifold with the Riemannian metric
$$ds^{2}=\sum g_{ij}dx^{i}dx^{j}$$ for the local coordinate system $\{x^{i}\}_{1\leq i\leq n}$, where $g=\det(g_{ij})$. Let $dV_{g}$ be the volume form associated to the metric $g$, and $\nabla_{g} f$ is the gradient with respect to the metric $g$. Let $K$ be the sectional
curvature on $M$. We say that $M$ has {\em negative curvature}, if $K\leq 0$ along every plane section at every point of $M$. Moreover, $M$
contains no points conjugate to any point $x_{0}$ of $M$. If $M$ is simply connected, then the exponential mapping
$$\exp_{x_{0}}:T_{x_{0}}M\rightarrow M$$
is a diffeomorphism, where $T_{x_{0}}M$ is the tangent space to $M$ at a point $x_{0}$.

We will work on complete, simply connected Riemannian manifold with negative curvature. Let $x_{0}\in M$. Then, $\rho(x)={\rm dist}(x,x_{0})$ is
smooth on $M\backslash\{x_{0}\}$, and satisfies the condition
$$|\nabla_{g} \rho(x)|=1,\;\;x\in M\backslash\{x_{0}\},$$
where ${\rm dist}(\cdot,\cdot)$ is the geodesic distance.

In particular, we will also work on the Poincar\'{e} ball model (coordinate map) of the hyperbolic space $\mathbb{H}^{n}$ ($n\geq2$), that
is, when $M$ has constant curvature equal to $-1$. This is the unit ball $B$ in $\Rn$ centered at the origin and equipped with the
Riemannian metric
$$ds^{2}=\frac{4\sum_{i=1}^{n}dx_{i}^{2}}{(1-|x|^{2})^{2}},$$
where $|\cdot|$ is the Euclidean distance.

The Riemannian measure, the gradient and the hyperbolic distance in the Poincar\'{e} ball model are, respectively,
$$dV_{g}=\frac{2^{n}}{(1-|x|^{2})^{n}}dx,$$
$$\nabla_{g}=\left(\frac{1-|x|^{2}}{2}\right)^{2}\nabla,$$
and
$$\rho(x)=\ln\frac{1+|x|}{1-|x|},$$
where $\nabla$ is the usual gradient, and $dx$ is the Lebesgue measure in $\Rn$.

We also use the polar coordinate change formula
\begin{equation}\label{polar}
\int_{\mathbb{H}^{n}}fdV_{g}=\int_{0}^{+\infty}\int_{\mathbb{S}^{n-1}}f\cdot (\sinh \rho)^{n-1}d\rho d\sigma
\end{equation}
for $f\in L^{1}(\mathbb{H}^{n})$, where $\mathbb{S}^{n-1}$ is the unit sphere in $\mathbb{H}^{n}$.

The Sobolev space $W_{0}^{1,n}(M)$ is defined as the completion of $C_{0}^{\infty}(M)$ in the norm
$$\|f\|_{W^{1,n}(M)}=\left(\int_{M}(|\nabla_{g}f(x)|^{n}+|f(x)|^{n})dV_{g}\right)^{1/n}.$$

\section{Hardy type inequalities on manifolds}
\label{SEC:Hardy_manif}
In this section we prove a family of Hardy type inequalities on complete, simply connected Riemannian manifold $M$ with negative curvature.

Let us first recall the Moser-Trudinger inequality on $M$:
\begin{thm}[{\cite[Theorem 1.3]{DY16}}]
\label{Trudinger_manif_thm}
Let $M$ be a complete, simply connected Riemannian manifold of dimension $n\geq2$ with negative curvature. Let $0\leq\beta<n$ and let $0<\alpha\leq \alpha_{\beta}$ with $\alpha_{\beta}=n\omega_{n-1}^{1/(n-1)}(1-\beta/n)$. Then there exists a positive constant $\widetilde{C_{1}}=\widetilde{C_{1}}(\alpha, \beta,n,M)$ such that
\begin{equation}\label{Trudinger_manif1}
\int_{M}\frac{1}{\rho^{\beta}}\left(\exp(\alpha|f(x)|^{n/(n-1)})-\sum_{k=0}^{n-2}\frac{\alpha^{k}
|f(x)|^{kn/(n-1)}}{k!}\right)dV_{g}\leq \widetilde{C_{1}}
\end{equation}
holds for all functions $f\in W_{0}^{1,n}(M)$ with $\|f\|_{W^{1,n}(M)}\leq1$, where $\omega_{n-1}$ is the area of the surface of the unit $n$-ball in $M$. Moreover, the constant $\alpha_{\beta}$ is sharp.
\end{thm}
Now we give our result on the Hardy inequalities, and on their equivalence with \eqref{Trudinger_manif1} when $0<\alpha<\alpha_{\beta}$.
\begin{thm}\label{Hardy_manif_thm}
Let $M$ be a complete, simply connected Riemannian manifold of dimension $n\geq2$ with negative curvature. Let $0\leq\beta<n$. Then for any $n\leq q< \infty$ there exists a positive constant $C_{1}=C_{1}(n,\beta, q, M)$ such that
\begin{equation}\label{Hardy_manif1}
\left\|\frac{f}{\rho^{\frac{\beta}{q}}}\right\|_{L^{q}(M)}\leq C_{1}q^{1-1/n}\|f\|_{W^{1,n}(M)}
\end{equation}
holds for all functions $f\in W_{0}^{1,n}(M)$. Moreover, we have
\begin{equation}\label{equiv_identity_Hardy}
\frac{1}{\alpha_{\beta} n'e}=A_{1}^{n'}=B_{1}^{n'},
\end{equation}
where
$$\alpha_{\beta}=n\omega_{n-1}^{1/(n-1)}(1-\beta/n),$$
\begin{equation*}
\begin{split}
A_{1}=\inf\{C_{1}>0; \exists r=r(n,&\beta, C_{1}) \textrm{ with } r\geq n:\\& \eqref{Hardy_manif1}\textrm{ holds }\forall f\in
W_{0}^{1,n}(M),
\forall q
\textrm{ with } r\leq q<\infty\},
\end{split}
\end{equation*}
\begin{equation}\label{alphaAB_Hardy}
B_{1}=\limsup_{q\rightarrow \infty}\sup_{f\in W_{0}^{1,n}(M)\backslash\{0\}}
\frac{\left\|\frac{f}{\rho^{\frac{\beta}{q}}}\right\|_{L^{q}(M)}}{q^{1-1/n}\|f\|_{W^{1,n}(M)}}.
\end{equation}
The weighted Trudinger-Moser inequalities \eqref{Trudinger_manif1} with $0<\alpha<\alpha_{\beta}$ are equivalent to the Hardy type inequalities \eqref{Hardy_manif1} with relation \eqref{equiv_identity_Hardy}.
\end{thm}
\begin{rem}\label{rem_B1} By \eqref{equiv_identity_Hardy} and \eqref{alphaAB_Hardy}, we see that the constant
$$B_{1}=(n\omega_{n-1}^{1/(n-1)}(1-\beta/n)n'e)^{-1/n'}$$
is asymptotically sharp for \eqref{Hardy_manif1}, i.e. \eqref{Hardy_manif1} does not hold for $0<C_{1}<B_{1}$.

In fact, \eqref{Hardy_manif1} implies \eqref{Trudinger_manif1} for $0<\alpha<\widehat{\alpha}$ for some $\widehat{\alpha}>0$, while \eqref{Hardy_manif1} and \eqref{equiv_identity_Hardy} together imply \eqref{Trudinger_manif1} for all $0<\alpha<\alpha_{\beta}$. The same remark applies to Theorems \ref{Hardy_hyper_thm}, \ref{GN_hyper_thm} and \ref{GN3_hyper_thm}.
\end{rem}
\begin{proof}[Proof of Theorem \ref{Hardy_manif_thm}] Since $B_{1}\leq A_{1}$, in order to obtain \eqref{equiv_identity_Hardy} it is enough to show that \eqref{Trudinger_manif1}$\Rightarrow$\eqref{Hardy_manif1} with $\alpha_{\beta}\leq(en'A_{1}^{n'})^{-1}$ and
\eqref{Hardy_manif1}$\Rightarrow$\eqref{Trudinger_manif1} with $1/\alpha_{\beta}\leq n'eB_{1}^{n'}$. Let us start to prove \eqref{Trudinger_manif1}$\Rightarrow$\eqref{Hardy_manif1} with $\alpha_{\beta}\leq(en'A_{1}^{n'})^{-1}$. In the case $\|f\|_{W^{1,n}(M)}=0$ taking into account the definition of $f\in W_{0}^{1,n}(M)$ we have $f\equiv0$, that is, \eqref{Hardy_manif1} is trivial. Therefore, we can assume that $\|f\|_{W^{1,n}(M)}\neq0$. Replacing $f$ by $f/\|f\|_{W^{1,n}(M)}$ in \eqref{Trudinger_manif1} with $0<\alpha<\alpha_{\beta}$ we get
\begin{equation}\label{Hardy_manif3_0}
\int_{M}\frac{1}{\rho^{\beta}}\sum_{k=n-1}^{\infty}
\frac{\alpha^{k}|f(x)|^{kn'}}{k!\|f\|_{W^{1,n}(M)}^{kn'}}dV_{g}\leq \widetilde{C_{1}}.
\end{equation}
It implies that for any $\varepsilon$ with $0<\varepsilon<\alpha_{\beta}$ there exists $C_{\varepsilon}$ such that
\begin{equation}\label{Hardy_manif3}
\int_{M}\frac{1}{\rho^{\beta}}\sum_{k=n-1}^{\infty}
\frac{(\alpha_{\beta}-\varepsilon)^{k}|f(x)|^{kn'}}{k!\|f\|_{W^{1,n}(M)}^{kn'}}dV_{g}\leq C_{\varepsilon}.
\end{equation}
In particular, it follows that
\begin{equation}\label{Hardy_manif4}
\left\|\frac{f}{\rho^{\frac{\beta}{kn'}}}\right\|_{L^{kn'}(M)}\leq (C_{\varepsilon}k!)^{1/(kn')}
(\alpha_{\beta}-\varepsilon)^{-1/n'}\|f\|_{W^{1,n}(M)}
\end{equation}
for all $k\geq n-1$. Moreover, for any $q\geq n$, there exists an integer $k\geq n-1$ satisfying $n'k\leq q <n'(k+1)$. Then, using H\"{o}lder's inequality for $\frac{\theta q}{n'k}+\frac{(1-\theta)
q}{n'(k+1)}=1$ with $0<\theta\leq1$ we calculate
\begin{equation*}
\begin{split}
\int_{M}\frac{|f(x)|^{q}}{\rho^{\beta}}dV_{g}&=\int_{M}\frac{|f(x)|^{\theta q}}{\rho^{\frac{\beta\theta
q}{n'k}}}\cdot\frac{|f(x)|^{(1-\theta)q}}{\rho^{\frac{\beta(1-\theta) q}{n'(k+1)}}}dV_{g}\\&
\leq \left(\int_{M}\frac{|f(x)|^{n'k}}{\rho^{\beta}}dV_{g}\right)^{\frac{\theta q}{n'k}}
\left(\int_{M}\frac{|f(x)|^{n'(k+1)}}{\rho^{\beta}}dV_{g}\right)^{\frac{(1-\theta) q}{n'(k+1)}}\\&
=\left\|\frac{f}{\rho^{\frac{\beta}{n'k}}}\right\|_{L^{n'k}(M)}^{\theta q}
\left\|\frac{f}{\rho^{\frac{\beta}{n'(k+1)}}}\right\|_{L^{n'(k+1)}(M)}^{(1-\theta)q},
\end{split}
\end{equation*}
that is,
\begin{equation}\label{Hardy_manif5}
\left\|\frac{f}{\rho^{\frac{\beta}{q}}}\right\|_{L^{q}(M)}
\leq\left\|\frac{f}{\rho^{\frac{\beta}{n'k}}}\right\|_{L^{n'k}(M)}^{\theta}
\left\|\frac{f}{\rho^{\frac{\beta}{n'(k+1)}}}\right\|_{L^{n'(k+1)}(M)}^{1-\theta}.
\end{equation}
Combining this with \eqref{Hardy_manif4}, we obtain
\begin{equation}\label{Hardy_manif6}
\left\|\frac{f}{\rho^{\frac{\beta}{q}}}\right\|_{L^{q}(M)}\leq C_{\varepsilon}^{\frac{1}{q}}(\alpha_{\beta}-\varepsilon)^{-\frac{1}{n'}}
((k+1)!)^{\frac{1}{q}}\|f\|_{W^{1,n}(M)}.
\end{equation}
Since $q\geq n'k$ we have $(k+1)!\leq \Gamma(q/n'+2)$, then \eqref{Hardy_manif6} implies that
\begin{equation}\label{Hardy_manif7}
\left\|\frac{f}{\rho^{\frac{\beta}{q}}}\right\|_{L^{q}(M)}\leq
(C_{\varepsilon}\Gamma(q/n'+2))^{1/q}(\alpha_{\beta}-\varepsilon)^{-1/n'}
\|f\|_{W^{1,n}(M)}
\end{equation}
for any $q\geq n$ and for all $f\in W_{0}^{1,n}(M)$, which is \eqref{Hardy_manif1}.
Now applying the Stirling formula for $q\rightarrow+\infty$, one gets
\begin{equation}\label{Gamma1}
\begin{split}
\Gamma(q/n'+2)^{1/q}&=\left((1+o(1))\sqrt{2\pi\left(q/n'+1\right)}\left(\frac{q/n'+1}{e}\right)^{q/n'+1}\right)^{1/q}
\\&=(1+o(1))\left(\frac{q}{en'}\right)^{1/n'}.
\end{split}
\end{equation}
Combining this with \eqref{Hardy_manif7}, we have as $q\rightarrow+\infty$, asymptotically
\begin{equation*}
\left\|\frac{f}{\rho^{\frac{\beta}{q}}}\right\|_{L^{q}(M)}\leq
(1+o(1))\left(\frac{q}{en'(\alpha_{\beta}-\varepsilon)}\right)^{1/n'}\|f\|_{W^{1,n}(M)},
\end{equation*}
that is, for any $\delta>0$ there exists $r\geq n$ such that
\begin{equation}\label{Hardy_manif8}
\left\|\frac{f}{\rho^{\frac{\beta}{q}}}\right\|_{L^{q}(M)}\leq
((n'e(\alpha_{\beta}-\varepsilon))^{-1/n'}+\delta)q^{1-1/n}\|f\|_{W^{1,n}(M)}
\end{equation}
holds for all $f\in W_{0}^{1,n}(M)$ and all $q$ with $r\leq q<\infty$.

Thus, we see that $A_{1}\leq (n'e(\alpha_{\beta}-\varepsilon))^{-1/n'}+\delta$, then by the arbitrariness of $\varepsilon$ and
$\delta$ we obtain $\alpha_{\beta}\leq(en'A_{1}^{n'})^{-1}$.

Now we show that \eqref{Hardy_manif1}$\Rightarrow$\eqref{Trudinger_manif1} with $1/\alpha_{\beta}\leq n'eB_{1}^{n'}$. By
\eqref{Hardy_manif1}, for any $q$ with $n\leq q<\infty$ there is $C_{1}=C_{1}(n,\beta,q,M)>0$ such that
\begin{equation}\label{Hardy_manif9}
\left\|\frac{f}{\rho^{\frac{\beta}{q}}}\right\|_{L^{q}(M)}\leq C_{1}q^{1-1/n}\|f\|_{W^{1,n}(M)}
\end{equation}
holds for all $f\in W_{0}^{1,n}(M)$. With the help of this and $\|f\|_{W^{1,n}(M)}\leq1$, we write
\begin{equation}\label{Hardy_manif10}
\int_{M}\frac{1}{\rho^{\beta}}\left(\exp(\alpha|f(x)|^{n'})-\sum_{k=0}^{n-2}
\frac{1}{k!}(\alpha|f(x)|^{n'})^{k}\right)dV_{g}
\leq \sum_{n'k\geq n,\;k\in\mathbb{N}}\frac{(\alpha n'kC_{1}^{n'})^{k}}{k!}.
\end{equation}
The series in the right hand side of \eqref{Hardy_manif10} converges when $0\leq\alpha<1/(n'eC_{1}^{n'})$. Thus, we have obtained \eqref{Trudinger_manif1} with $0\leq\alpha<1/(n'eC_{1}^{n'})$. Hence $\alpha_{\beta}\geq 1/(n'eC_{1}^{n'})$ for all $C_{1}\geq B_{1}$, which gives $\alpha_{\beta}\geq 1/(n'eB_{1})^{n'}$.

Thus, we have completed the proof of Theorem \ref{Hardy_manif_thm}.
\end{proof}
\section{Hardy type inequalities on hyperbolic spaces}
\label{SEC:Hardy_hyper}
In this section we show Hardy type inequalities with sharp constants on hyperbolic spaces and prove their equivalence with the Trudinder-Moser inequalities.
Let us start by recalling the following result on $\mathbb{H}^{n}\;(n\geq2)$:
\begin{thm}[{\cite[Theorem 1.1]{Zhu15}}]
\label{Trudinger_hyper_thm}
Let $\mathbb{H}^{n}\;(n\geq2)$ be the $n$-dimensional hyperbolic space. Let $0\leq\beta<n$ and let $0<\alpha\leq \alpha_{\beta}$ with $\alpha_{\beta}=n\omega_{n-1}^{1/(n-1)}(1-\beta/n)$. Then there
exists a positive constant $\widetilde{C_{2}}=\widetilde{C_{2}}(\alpha, \beta, n)$ such that
\begin{equation}\label{Trudinger_hyper1}
\int_{\mathbb{H}^{n}}\frac{1}{\rho^{\beta}}\left(\exp(\alpha|f(x)|^{n/(n-1)})-\sum_{k=0}^{n-2}\frac{\alpha^{k}
|f(x)|^{kn/(n-1)}}{k!}\right)dV_{g}\leq \widetilde{C_{2}}
\end{equation}
holds for all functions $f\in W_{0}^{1,n}(\mathbb{H}^{n})$
with $\|\nabla_{g} f\|_{L^{n}(\mathbb{H}^{n})}\leq1$, where $\omega_{n-1}$ is the area of the surface of the unit $n$-ball in $\mathbb{H}^{n}$. Furthermore, the constant $\alpha_{\beta}$ is sharp.
\end{thm}
We now show that this is equivalent to the following Hardy inequality.
\begin{thm}\label{Hardy_hyper_thm}
Let $\mathbb{H}^{n}\;(n\geq2)$ be the $n$-dimensional hyperbolic space and let $0\leq\beta<n$. Then for any $n\leq q <\infty$ there exists a
positive constant $C_{2}=C_{2}(n,\beta,q)$ such that
\begin{equation}\label{Hardy_hyper1}
\left\|\frac{f}{\rho^{\frac{\beta}{q}}}\right\|_{L^{q}(\mathbb{H}^{n})}\leq C_{2}q^{1-1/n}\|\nabla_{g} f\|_{L^{n}(\mathbb{H}^{n})}
\end{equation}
holds for all functions $f\in W_{0}^{1,n}(\mathbb{H}^{n})$. Furthermore, we have
\begin{equation}\label{equiv_identity_Hardy_hyper}
\frac{1}{\alpha_{\beta} n'e}=A_{2}^{n'}=B_{2}^{n'},
\end{equation}
where
$$\alpha_{\beta}=n\omega_{n-1}^{1/(n-1)}(1-\beta/n),$$
\begin{equation*}
\begin{split}
A_{2}=\inf\{C_{2}>0; \exists r=r(n,&\beta,C_{2}) \textrm{ with } r\geq n:\\& \eqref{Hardy_hyper1}\textrm{ holds }\forall f\in
W_{0}^{1,n}(\mathbb{H}^{n}),
\forall q
\textrm{ with } r\leq q<\infty\},
\end{split}
\end{equation*}
\begin{equation}\label{alphaDF_Hardy}
B_{2}=\limsup_{q\rightarrow \infty}\sup_{f\in W^{1,n}(\mathbb{H}^{n})\backslash\{0\}}
\frac{\left\|\frac{f}{\rho^{\frac{\beta}{q}}}\right\|_{L^{q}(\mathbb{H}^{n})}}{q^{1-1/n}\|\nabla_{g} f\|_{L^{n}(\mathbb{H}^{n})}}.
\end{equation}
The weighted Trudinger-Moser inequalities \eqref{Trudinger_hyper1} with $0<\alpha<\alpha_{\beta}$ are equivalent to the Hardy type inequalities \eqref{Hardy_hyper1} with relation \eqref{equiv_identity_Hardy_hyper}.
\end{thm}
\begin{rem}\label{rem_B2} An analogue of Remark \ref{rem_B1} holds, in particular, $B_{2}$ is asymptotically sharp for \eqref{Hardy_hyper1}.
\end{rem}
The proof is similar to that of Theorem \ref{Hardy_manif_thm} but we give it here for clarity.
\begin{proof}[Proof of Theorem \ref{Hardy_hyper_thm}] Since $B_{2}\leq A_{2}$, in order to obtain \eqref{equiv_identity_Hardy_hyper} it suffices to show that \eqref{Trudinger_hyper1}$\Rightarrow$\eqref{Hardy_hyper1} with $\alpha_{\beta}\leq(en'A_{2}^{n'})^{-1}$ and
\eqref{Hardy_hyper1}$\Rightarrow$\eqref{Trudinger_hyper1} with $1/\alpha_{\beta}\leq n'eB_{2}^{n'}$. We first show that \eqref{Trudinger_hyper1}$\Rightarrow$\eqref{Hardy_hyper1} with $\alpha_{\beta}\leq(en'A_{2}^{n'})^{-1}$. The case $\|\nabla_{g} f\|_{L^{n}(\mathbb{H}^{n})}=0$ is trivial, since we have $f\equiv0$ by the definition of $f\in W_{0}^{1,n}(\mathbb{H}^{n})$. Therefore, we can replace $f$ by $f/\|\nabla_{g} f\|_{L^{n}(\mathbb{H}^{n})}$ in \eqref{Trudinger_hyper1} with $0<\alpha<\alpha_{\beta}$ to get
\begin{equation}\label{Hardy_hyper3_0}
\int_{\mathbb{H}^{n}}\frac{1}{\rho^{\beta}}\sum_{k=n-1}^{\infty}
\frac{\alpha^{k}|f(x)|^{kn'}}{k!\|\nabla_{g} f\|_{L^{n}(\mathbb{H}^{n})}^{kn'}}dV_{g}\leq \widetilde{C_{2}}.
\end{equation}
In other words, it means that there exists $C_{\varepsilon}$ for any $\varepsilon$ with $0<\varepsilon<\alpha_{\beta}$ such that
\begin{equation}\label{Hardy_hyper3}
\int_{\mathbb{H}^{n}}\frac{1}{\rho^{\beta}}\sum_{k=n-1}^{\infty}
\frac{(\alpha_{\beta}-\varepsilon)^{k}|f(x)|^{kn'}}{k!\|\nabla_{g} f\|_{L^{n}(\mathbb{H}^{n})}^{kn'}}dV_{g}\leq C_{\varepsilon}.
\end{equation}
In particular, it follows that
\begin{equation}\label{Hardy_hyper4}
\left\|\frac{f}{\rho^{\frac{\beta}{kn'}}}\right\|_{L^{kn'}(\mathbb{H}^{n})}\leq (C_{\varepsilon}k!)^{1/(kn')}
(\alpha_{\beta}-\varepsilon)^{-1/n'}\|\nabla_{g} f\|_{L^{n}(\mathbb{H}^{n})}
\end{equation}
for all $k\geq n-1$. Moreover, for any $q\geq n$, there exists an integer $k\geq n-1$ satisfying $n'k\leq q <n'(k+1)$. Then, applying H\"{o}lder's inequality for $\frac{\theta q}{n'k}+\frac{(1-\theta)
q}{n'(k+1)}=1$ with $0<\theta\leq1$ one calculates
\begin{equation}\label{Hardy_hyper41}
\begin{split}
\int_{\mathbb{H}^{n}}\frac{|f(x)|^{q}}{\rho^{\beta}}dV_{g}&=\int_{\mathbb{H}^{n}}\frac{|f(x)|^{\theta q}}{\rho^{\frac{\beta\theta
q}{n'k}}}\cdot\frac{|f(x)|^{(1-\theta)q}}{\rho^{\frac{\beta(1-\theta) q}{n'(k+1)}}}dV_{g}\\&
\leq \left(\int_{\mathbb{H}^{n}}\frac{|f(x)|^{n'k}}{\rho^{\beta}}dV_{g}\right)^{\frac{\theta q}{n'k}}
\left(\int_{\mathbb{H}^{n}}\frac{|f(x)|^{n'(k+1)}}{\rho^{\beta}}dV_{g}\right)^{\frac{(1-\theta) q}{n'(k+1)}}\\&
=\left\|\frac{f}{\rho^{\frac{\beta}{n'k}}}\right\|_{L^{n'k}(\mathbb{H}^{n})}^{\theta q}
\left\|\frac{f}{\rho^{\frac{\beta}{n'(k+1)}}}\right\|_{L^{n'(k+1)}(\mathbb{H}^{n})}^{(1-\theta)q},
\end{split}
\end{equation}
which implies that
\begin{equation}\label{Hardy_hyper5}
\left\|\frac{f}{\rho^{\frac{\beta}{q}}}\right\|_{L^{q}(\mathbb{H}^{n})}
\leq\left\|\frac{f}{\rho^{\frac{\beta}{n'k}}}\right\|_{L^{n'k}(\mathbb{H}^{n})}^{\theta}
\left\|\frac{f}{\rho^{\frac{\beta}{n'(k+1)}}}\right\|_{L^{n'(k+1)}(\mathbb{H}^{n})}^{1-\theta}.
\end{equation}
We can combine this with \eqref{Hardy_hyper4} to derive that
\begin{equation}\label{Hardy_hyper6}
\left\|\frac{f}{\rho^{\frac{\beta}{q}}}\right\|_{L^{q}(\mathbb{H}^{n})}\leq
C_{\varepsilon}^{\frac{1}{q}}(\alpha_{\beta}-\varepsilon)^{-\frac{1}{n'}}
((k+1)!)^{\frac{1}{q}}\|\nabla_{g} f\|_{L^{n}(\mathbb{H}^{n})},
\end{equation}
that is,
\begin{equation}\label{Hardy_hyper7}
\left\|\frac{f}{\rho^{\frac{\beta}{q}}}\right\|_{L^{q}(\mathbb{H}^{n})}\leq
(C_{\varepsilon}\Gamma(q/n'+2))^{1/q}(\alpha_{\beta}-\varepsilon)^{-1/n'}
\|\nabla_{g} f\|_{L^{n}(\mathbb{H}^{n})}
\end{equation}
for any $q\geq n$ and for all $f\in W_{0}^{1,n}(\mathbb{H}^{n})$, which is \eqref{Hardy_hyper1}, where we have used $(k+1)!\leq \Gamma(q/n'+2)$ when $q\geq n'k$. Now taking into account the behavior of $\Gamma(q/n'+2)$ for $q\rightarrow+\infty$ by \eqref{Gamma1}, \eqref{Hardy_hyper7} gives that for any $\delta>0$ there exists $r\geq n$ such that
\begin{equation}\label{Hardy_hyper8}
\left\|\frac{f}{\rho^{\frac{\beta}{q}}}\right\|_{L^{q}(\mathbb{H}^{n})}\leq
((n'e(\alpha_{\beta}-\varepsilon))^{-1/n'}+\delta)q^{1-1/n}\|\nabla_{g} f\|_{L^{n}(\mathbb{H}^{n})}
\end{equation}
holds for all $f\in W_{0}^{1,n}(\mathbb{H}^{n})$ and all $q$ with $r\leq q<\infty$.

Thus, we get $A_{2}\leq (n'e(\alpha_{\beta}-\varepsilon))^{-1/n'}+\delta$. Since $\varepsilon$ and
$\delta$ are arbitrary, it implies that $\alpha_{\beta}\leq(en'A_{2}^{n'})^{-1}$.

It remains to show that \eqref{Hardy_hyper1}$\Rightarrow$\eqref{Trudinger_hyper1} with $1/\alpha_{\beta}\leq n'eB_{2}^{n'}$. Since we have \eqref{Hardy_hyper1}, we can write that for any $q$ with $n\leq q<\infty$ there is $C_{2}=C_{2}(n,\beta,q)>0$ such that
\begin{equation}\label{Hardy_hyper9}
\left\|\frac{f}{\rho^{\frac{\beta}{q}}}\right\|_{L^{q}(\mathbb{H}^{n})}\leq C_{2}q^{1-1/n}\|\nabla_{g}
f\|_{L^{n}(\mathbb{H}^{n})}
\end{equation}
holds for all $f\in W_{0}^{1,n}(\mathbb{H}^{n})$.
Employing this and $\|\nabla_{g} f\|_{L^{n}(\mathbb{H}^{n})}\leq1$, we get
\begin{equation}\label{Hardy_hyper10}
\int_{\mathbb{H}^{n}}\frac{1}{\rho^{\beta}}\left(\exp(\alpha|f(x)|^{n'})-\sum_{k=0}^{n-2}
\frac{1}{k!}(\alpha|f(x)|^{n'})^{k}\right)dV_{g}\leq\sum_{n'k\geq n,\;k\in\mathbb{N}}\frac{(\alpha n'kC_{2}^{n'})^{k}}{k!}.
\end{equation}
The series in the right hand side of \eqref{Hardy_hyper10} converges when $0\leq\alpha<1/(n'eC_{2}^{n'})$. Thus, we have obtained \eqref{Trudinger_hyper1} with $0\leq\alpha<1/(n'eC_{2}^{n'})$. Hence, $\alpha_{\beta}\geq 1/(n'eC_{2})^{n'}$ for all $C_{2}\geq B_{2}$, that is, $\alpha_{\beta}\geq 1/(n'eB_{2}^{n'})$. This completes the proof of Theorem \ref{Hardy_hyper_thm}.
\end{proof}
Now let us show the corresponding uncertainty type principle on hyperbolic spaces.
\begin{thm}\label{uncer_thm} Let $\mathbb{H}^{n}\;(n\geq2)$ be the $n$-dimensional hyperbolic space and let $0\leq\beta<n$. Then we have
\begin{multline}\label{uncer_1}
\left(\int_{\mathbb{H}^{n}}|\nabla_{g}
f(x)|^{n}dV_{g}\right)^{1/n}\left(\int_{\mathbb{H}^{n}}\rho^{q'}|f(x)|^{q'}dV_{g}\right)^{1/q'}\\ \geq
C_{2}^{-1}q^{1/n-1}\int_{\mathbb{H}^{n}}\rho^{\frac{q-\beta}{q}}|f(x)|^{2}dV_{g}
\end{multline}
for all functions $f\in W_{0}^{1,n}(\mathbb{H}^{n})$, where $1/q+1/q'=1$, and $C_{2}$ is the constant from \eqref{Hardy_hyper1}.
\end{thm}
\begin{proof}[Proof of Theorem \ref{uncer_thm}] Using \eqref{Hardy_hyper1} and H\"{o}lder's inequality, we calculate
\begin{equation*}
\begin{split}
\left(\int_{\mathbb{H}^{n}}|\nabla_{g} f(x)|^{n}dV_{g}\right)^{1/n}&\left(\int_{\mathbb{H}^{n}}\rho^{q'}|f(x)|^{q'}dx\right)^{1/q'}\\&
\geq C_{2}^{-1}q^{1/n-1}\left(\int_{\mathbb{H}^{n}}\frac{|f(x)|^{q}}{\rho^{\beta}}dV_{g}\right)^{1/q}
\left(\int_{\mathbb{H}^{n}}\rho^{q'}|f(x)|^{q'}dV_{g}\right)^{1/q'}\\&
\geq C_{2}^{-1}q^{1/n-1}\int_{\mathbb{H}^{n}}\rho^{\frac{q-\beta}{q}}|f(x)|^{2}dV_{g},
\end{split}
\end{equation*}
which gives \eqref{uncer_1}.
\end{proof}
\section{Caffarelli-Kohn-Nirenberg inequalities on hyperbolic spaces}
\label{SEC:CKN}
In this section we give new Caffarelli-Kohn-Nirenberg inequalities on hyperbolic spaces, and show their equivalence with the weighted
Trudinger-Moser inequalities.

Let us first show that the obtained Hardy inequalities in turn imply the following Caffarelli-Kohn-Nirenberg type inequalities on hyperbolic spaces.
\begin{thm}\label{CKN_thm} Let $\mathbb{H}^{n}\;(n\geq2)$ be the $n$-dimensional hyperbolic space. Let $b$, $c\in\mathbb{R}$, $0<p_{3}<\infty$ and $1<p_{2}<\infty$. Let $\delta\in(0,1]\cap\left(\frac{p_{3}-p_{2}}{p_{3}},1\right]$. Let $0\leq b(1-\delta)-c<n(1/p_{3}-(1-\delta)/p_{2})$ and $n\leq\frac{\delta
p_{2}p_{3}}{p_{2}-(1-\delta)p_{3}}$. Then we have
\begin{equation}\label{CKN_1}
\|\rho^{c}f\|_{L^{p_{3}}(\mathbb{H}^{n})}
\leq \widehat{C_{3}}\|\nabla_{g} f\|^{\delta}_{L^{n}(\mathbb{H}^{n})}
\|\rho^{b}f\|^{1-\delta}_{L^{p_{2}}(\mathbb{H}^{n})},
\end{equation}
for all functions $f\in W_{0}^{1,n}(\mathbb{H}^{n})$, where \\$\widehat{C_{3}}=C_{2}^{\delta}\left(\frac{\delta p_{2}p_{3}}
{p_{2}-(1-\delta)p_{3}}\right)^{\delta-\frac{\delta}{n}}$, and $C_{2}$ is the constant from
\eqref{Hardy_hyper1}.
\end{thm}
\begin{proof}[Proof of Theorem \ref{CKN_thm}]
{\bf Case $\delta=1$}. In this case, we have $0\leq-c<n/p_{3}$, $n\leq p_{3}<\infty$ and $\widehat{C_{3}}=C_{2}p_{3}^{1-1/n}$, so \eqref{CKN_1} is equivalent to
\eqref{Hardy_hyper1}.

{\bf Case $\delta\in(0,1)\cap\left(\frac{p_{3}-p_{2}}{p_{3}},1\right)$}.
Using H\"{o}lder's inequality for $\frac{
p_{2}-(1-\delta)p_{3}}{p_{2}}+\frac{(1-\delta)p_{3}}{p_{2}}=1$, we calculate
\begin{equation}\label{CKN_cal}
\begin{split}
\|\rho^{c}f\|_{L^{p_{3}}(\mathbb{H}^{n})}&=
\left(\int_{\mathbb{H}^{n}}\rho^{cp_{3}}|f(x)|^{p_{3}}dV_{g}\right)^{\frac{1}{p_{3}}}\\&
=\left(\int_{\mathbb{H}^{n}}\frac{|f(x)|^{\delta p_{3}}}{\rho^{\delta p_{3}\left(\frac{b(1-\delta)-c}{\delta}\right)}}\cdot
\frac{|f(x)|^{(1-\delta)p_{3}}}{\rho^{-bp_{3}(1-\delta)}}dV_{g}\right)^{\frac{1}{p_{3}}}
\\&
\leq
\left(\int_{\mathbb{H}^{n}}\frac{|f(x)|^{\frac{\delta p_{2}p_{3}}{p_{2}-(1-\delta)p_{3}}}}{\rho^{\frac{b(1-\delta)-c}{\delta}
\cdot\frac{\delta p_{2}p_{3}}{p_{2}-(1-\delta)p_{3}}}}dV_{g}\right)^{\frac{p_{2}-(1-\delta)p_{3}}{{p_{2}p_{3}}}}
\left(\int_{\mathbb{H}^{n}}\frac{|f(x)|^{p_{2}}}{\rho^{-bp_{2}}}dV_{g}\right)^{\frac{1-\delta}{p_{2}}}\\&
=\left\|\frac{f}{\rho^{\frac{b(1-\delta)-c}{\delta}}}\right\|^{\delta}_{L^{\frac{\delta p_{2}p_{3}}{p_{2}-(1-\delta)p_{3}}}(\mathbb{H}^{n})}
\left\|\frac{f}{\rho^{-b}}\right\|^{1-\delta}_{L^{p_{2}}(\mathbb{H}^{n})}.
\end{split}
\end{equation}
Since we have $\delta>\frac{p_{3}-p_{2}}{p_{3}}$, that is, $n\leq\frac{\delta
p_{2}p_{3}}{p_{2}-(1-\delta)p_{3}}<\infty$, and $0\leq\frac{b(1-\delta)-c}{\delta}<\frac{n(p_{2}-(1-\delta)p_{3})}{\delta
p_{2}p_{3}}$, then using \eqref{Hardy_hyper1} in \eqref{CKN_cal} we obtain the desired inequality \eqref{CKN_1}.
\end{proof}

Now we show other types of Caffarelli-Kohn-Nirenberg inequalities with sharp constants, which are equivalent to Moser-Trudinger inequalities. First, let us start by recalling the following Moser-Trudinger inequality on $\mathbb{H}^{n}\;(n\geq2)$:
\begin{thm}[{\cite[Theorem 1.3]{LT13}}]
\label{Trudinger_GN_hyper_thm}
Let $\mathbb{H}^{n}\;(n\geq2)$ be the $n$-dimensional hyperbolic space. Let $0\leq\beta<n$ and let $0<\alpha<\alpha_{\beta}$ with $\alpha_{\beta}=n\omega_{n-1}^{1/(n-1)}(1-\beta/n)$. Then there exists a positive constant
$\widehat{C_{4}}=\widehat{C_{4}}(n, \alpha, \beta)$ such that
\begin{multline}\label{Trudinger_GN_hyper1}
\int_{\mathbb{H}^{n}}\frac{1}{\rho^{\beta}}\left(\exp(\alpha|f(x)|^{n/(n-1)})-\sum_{k=0}^{n-2}\frac{
\alpha^{k}|f(x)|^{kn/(n-1)}}{k!}\right)dV_{g}\leq \widehat{C_{4}}\int_{\mathbb{H}^{n}}
\frac{|f(x)|^{n}}{\rho^{\beta}}dV_{g}
\end{multline}
holds for all functions $f\in W_{0}^{1,n}(\mathbb{H}^{n})$ with $\|\nabla_{g}
f\|_{L^{n}(\mathbb{H}^{n})}\leq1$, where $\omega_{n-1}$ is the area of the surface of the unit $n$-ball in $\mathbb{H}^{n}$. Moreover, the constant $\alpha_{\beta}$ is sharp.
\end{thm}
First, we establish an extension of this result allowing weights of different orders.
\begin{thm}\label{analog_LT13}
Let $\mathbb{H}^{n}\;(n\geq2)$ be the $n$-dimensional hyperbolic space. Let $0\leq\beta_{1}<n$ and $\beta_{2}\in \mathbb{R}$. Let $0<\alpha<\alpha_{\beta_{1}}$ with $\alpha_{\beta_{1}}=n\omega_{n-1}^{1/(n-1)}(1-\beta_{1}/n)$. Let \\\begin{equation}\label{delta}
\delta=
\begin{cases} 0, \text{\;if}\; \beta_{1}=\beta_{2};\\
\delta:0\leq \beta_{1}-\beta_{2}(1-\delta)<n\delta\leq n, \text{\;if}\; \beta_{1}\neq\beta_{2}.\end{cases}
\end{equation} Then there exists a positive constant
$\widetilde{C_{3}}=\widetilde{C_{3}}(n, \alpha, \beta_{1}, \beta_{2}, \delta)$ such that
\begin{multline}\label{analog_LT13_1}
\int_{\mathbb{H}^{n}}\frac{1}{\rho^{\beta_{1}}}\left(\exp(\alpha|f(x)|^{n/(n-1)})-\sum_{k=0}^{n-2}\frac{
\alpha^{k}|f(x)|^{kn/(n-1)}}{k!}\right)dV_{g} \\\leq \widetilde{C_{3}}\left(\int_{\mathbb{H}^{n}}
\frac{|f(x)|^{n}}{\rho^{\beta_{2}}}dV_{g}\right)^{1-\delta}
\end{multline}
holds for all functions $f\in W_{0}^{1,n}(\mathbb{H}^{n})$ with $\|\nabla_{g}
f\|_{L^{n}(\mathbb{H}^{n})}\leq1$, where $\omega_{n-1}$ is the area of the surface of the unit $n$-ball in $\mathbb{H}^{n}$. Moreover, the constant $\alpha_{\beta_{1}}$ is sharp.
\end{thm}
\begin{rem}\label{rem_LT13} We note that the Theorem \ref{analog_LT13} implies Theorem \ref{Trudinger_GN_hyper_thm} when $\beta_{1}=\beta_{2}$.
\end{rem}
\begin{proof}[Proof of Theorem \ref{analog_LT13}] By \eqref{Trudinger_GN_hyper1} we have
\begin{multline}\label{analog_LT13_2}
\int_{\mathbb{H}^{n}}\frac{1}{\rho^{\beta_{1}}}\left(\exp(\alpha|f(x)|^{n/(n-1)})-\sum_{k=0}^{n-2}\frac{
\alpha^{k}|f(x)|^{kn/(n-1)}}{k!}\right)dV_{g}\leq \widehat{C_{4}}\int_{\mathbb{H}^{n}}
\frac{|f(x)|^{n}}{\rho^{\beta_{1}}}dV_{g}.
\end{multline}
When $0\leq\beta_{1}<n$ and $\beta_{2}\in\mathbb{R}$, we have by \eqref{CKN_1} with $p_{2}=p_{3}=n\geq2$, $b=-\beta_{2}/n$, $c=-\beta_{1}/n$ and $0\leq \beta_{1}-\beta_{2}(1-\delta)<n\delta\leq n$ that
\begin{equation}\label{analog_LT13_3}
\int_{\mathbb{H}^{n}}\frac{|f(x)|^{n}}{\rho^{\beta_{1}}}dV_{g}\leq \widehat{C_{3}}
\left(\int_{\mathbb{H}^{n}}\frac{|f(x)|^{n}}{\rho^{\beta_{2}}}dV_{g}\right)^{1-\delta}
\end{equation}
for all $f\in W_{0}^{1,n}(\mathbb{H}^{n})$ with $\|\nabla_{g}
f\|_{L^{n}(\mathbb{H}^{n})}\leq1$, where $\widehat{C_{3}}=\widehat{C_{3}}(n,\beta_{1},\beta_{2},\delta)$ is the constant from \eqref{CKN_1}. Then, by this we note that there exists a positive constant $C=C(n,\beta_{1},\beta_{2},\delta)$ such that
\begin{equation}\label{analog_LT13_3_8}
\int_{\mathbb{H}^{n}}\frac{|f(x)|^{n}}{\rho^{\beta_{1}}}dV_{g}\leq C
\left(\int_{\mathbb{H}^{n}}\frac{|f(x)|^{n}}{\rho^{\beta_{2}}}dV_{g}\right)^{1-\delta}
\end{equation}
holds for all $f\in W_{0}^{1,n}(\mathbb{H}^{n})$ with $\|\nabla_{g}
f\|_{L^{n}(\mathbb{H}^{n})}\leq1$, where $0\leq\beta_{1}<n$, $\beta_{2}\in\mathbb{R}$ and $\delta$ is given in \eqref{delta}. Then, combining \eqref{analog_LT13_2} and \eqref{analog_LT13_3_8}, we obtain \eqref{analog_LT13_1}.

Now we show the sharpness of the constant $\alpha_{\beta_{1}}$ in \eqref{analog_LT13_1}, that is, we prove that \eqref{analog_LT13_1} fails when $\alpha\geq\alpha_{\beta_{1}}$. By \cite{LT13}, we know that if we take the sequence $\{f_{j}\}_{j=1}^{\infty}\in W_{0}^{1,n}(\mathbb{H}^{n})$ as follows
\begin{equation*}
f_{j}(x):=\omega_{n-1}^{-1/n}D_{j}
\begin{cases} 0, \text{\;if}\;\rho>1;\\
j^{\frac{n-\beta_{1}-1}{n-\beta_{1}}}\left(\frac{-\ln\rho}{j}\right), \text{\;if}\;e^{-j}\leq\rho\leq1;\\
j^{\frac{n-\beta_{1}-1}{n-\beta_{1}}}, \text{\;if}\;
0\leq \rho \leq e^{-j},\end{cases}
\end{equation*}
where
$$D_{j}=\left(j^{-\frac{n}{n-\beta_{1}}}\int_{e^{-j}}^{1}\rho^{-n}(\sinh \rho)^{n-1}d\rho\right)^{-\frac{1}{n}},$$
then we have
$$D_{j}j^{-\frac{\beta_{1}}{n(n-\beta_{1})}}\rightarrow 1\;\;\text\;\;\text{as}\;\;j\rightarrow\infty,$$
$$\int_{\mathbb{H}^{n}}|\nabla_{g}f_{j}(x)|^{n}dV_{g}=1\;\;\text{and}\;\;\int_{\mathbb{H}^{n}}
\frac{|f_{j}(x)|^{n}}{\rho^{\beta_{2}}}dV_{g}=O\left(\frac{1}{j}\right).$$
Plugging $f_{j}(x)$ into the left hand side of \eqref{analog_LT13_1} and using the polar coordinates \eqref{polar}, we calculate
$$
\int_{\mathbb{H}^{n}}\frac{1}{\rho^{\beta_{1}}}\left(\exp(\alpha|f_{j}(x)|^{n/(n-1)})-\sum_{k=0}^{n-2}\frac{
\alpha^{k}|f_{j}(x)|^{kn/(n-1)}}{k!}\right)dV_{g}$$
$$=\int_{\mathbb{H}^{n}}\frac{1}{\rho^{\beta_{1}}}\sum_{k=n-1}^{\infty}\frac{
\alpha^{k}|f_{j}(x)|^{kn/(n-1)}}{k!}dV_{g}$$
$$ \geq \int_{\rho\leq e^{-j}}\frac{1}{\rho^{\beta_{1}}}\sum_{k=n-1}^{\infty}\frac{
\alpha^{k}|f_{j}(x)|^{kn/(n-1)}}{k!}dV_{g}$$
$$=\omega_{n-1}\sum_{k=n-1}^{\infty}\frac{
\alpha^{k}\left(\omega_{n-1}^{-1/n}D_{j}j^{\frac{n-\beta_{1}-1}{n-\beta_{1}}}\right)^{kn/(n-1)}}{k!}
\int_{0}^{e^{-j}}\frac{(\sinh \rho)^{n-1}}{\rho^{\beta_{1}}}d\rho$$
$$ \sim
\sum_{k=n-1}^{\infty}\frac{
\alpha^{k}\left(\omega_{n-1}^{-1/n}D_{j}j^{\frac{n-\beta_{1}-1}{n-\beta_{1}}}\right)^{kn/(n-1)}}{k!}
e^{-j(n-\beta_{1})}$$
$$
\sim \sum_{k=n-1}^{\infty}
\frac{\left(\frac{\alpha}{\omega_{n-1}^{1/(n-1)}}\right)^{k}\left(j^{\frac{\beta_{1}}{n(n-\beta_{1})}+
\frac{n-\beta_{1}-1}{n-\beta_{1}}}\right)^{\frac{kn}{n-1}}}{k!}
e^{-j(n-\beta_{1})}$$
$$ \sim
\sum_{k=n-1}^{\infty}
\frac{\left(\frac{\alpha jn}{n\omega_{n-1}^{1/(n-1)}}\right)^{k}}{k!}
e^{-j(n-\beta_{1})}$$
$$ \geq
\sum_{k=n-1}^{\infty}
\frac{j^{k}(n-\beta_{1})^{k}}{k!}
e^{-j(n-\beta_{1})}$$
$$=
1-\sum_{k=0}^{n-2}
\frac{j^{k}(n-\beta_{1})^{k}}{k!}
e^{-j(n-\beta_{1})}.
$$
Thus we obtain that
\begin{multline*}
\frac{\int_{\mathbb{H}^{n}}\frac{1}{\rho^{\beta_{1}}}\left(\exp(\alpha|f_{j}(x)|^{n/(n-1)})-\sum_{k=0}^{n-2}\frac{
\alpha^{k}|f_{j}(x)|^{kn/(n-1)}}{k!}\right)dV_{g}}
{\left(\int_{\mathbb{H}^{n}}
\frac{|f(x)|^{n}}{\rho^{\beta_{2}}}dV_{g}\right)^{1-\delta}}\\ \geq\frac{1-\sum_{k=0}^{n-2}
\frac{j^{k}(n-\beta_{1})^{k}}{k!}
e^{-j(n-\beta_{1})}}{O\left(\frac{1}{j}\right)}\rightarrow\infty
\end{multline*}
as $j\rightarrow\infty$.
\end{proof}

We now show that this is equivalent to the following Caffarelli-Kohn-Nirenberg type inequalities.
\begin{thm}\label{GN_hyper_thm}
Let $\mathbb{H}^{n}\;(n\geq2)$ be the $n$-dimensional hyperbolic space. Let $0\leq\beta_{1}<n$ and $\beta_{2}\in \mathbb{R}$. Let $\delta$ be as in \eqref{delta}. Then for any $n\leq q<\infty$ there exists a positive constant $C_{3}=C_{3}(n,\beta_{1},\beta_{2},q,\delta)$ such that
\begin{equation}\label{GN_hyper1}
\left\|\frac{f}{\rho^{\frac{\beta_{1}}{q}}}\right\|_{L^{q}(\mathbb{H}^{n})}\leq C_{3}q^{1-1/n}\|\nabla_{g}
f\|_{L^{n}(\mathbb{H}^{n})}^{1-\frac{n(1-\delta)}{q}}
\left\|\frac{f}{\rho^{\frac{\beta_{2}}{n}}}\right\|_{L^{n}(\mathbb{H}^{n})}^{\frac{n(1-\delta)}{q}}
\end{equation}
holds for all functions $f\in W_{0}^{1,n}(\mathbb{H}^{n})$. Moreover, we have
\begin{equation}\label{equiv_identity_GN_hyper}
\frac{1}{\alpha_{\beta_{1}} n'e}=A_{3}^{n'}=B_{3}^{n'},
\end{equation}
where
$$\alpha_{\beta_{1}}=n\omega_{n-1}^{1/(n-1)}(1-\beta_{1}/n),$$
\begin{equation*}
\begin{split}
A_{3}=\inf\{C_{3}>0; \exists r=r(n,\beta_{1},&\beta_{2},C_{3}) \textrm{ with } r\geq n:\\& \eqref{GN_hyper1}\textrm{ holds }\forall f\in
W_{0}^{1,n}(\mathbb{H}^{n}),
\forall q
\textrm{ with } r\leq q<\infty\},
\end{split}
\end{equation*}
\begin{equation}\label{alphaDF_GN}
B_{3}=\limsup_{q\rightarrow \infty}\sup_{f\in W^{1,n}(\mathbb{H}^{n})\backslash\{0\}}
\frac{\left\|\frac{f}{\rho^{\frac{\beta_{1}}{q}}}\right\|_{L^{q}(\mathbb{H}^{n})}}{q^{1-1/n}\|\nabla_{g}
f\|_{L^{n}(\mathbb{H}^{n})}^{1-\frac{n(1-\delta)}{q}}
\left\|\frac{f}{\rho^{\frac{\beta_{2}}{n}}}\right\|_{L^{n}(\mathbb{H}^{n})}^{\frac{n(1-\delta)}{q}}}.
\end{equation}
The weighted Trudinger-Moser inequalities \eqref{analog_LT13_1} are equivalent to the Caffarelli-Kohn-Nirenberg type inequalities \eqref{GN_hyper1} with relation \eqref{equiv_identity_GN_hyper}.
\end{thm}
\begin{rem}\label{rem_B3} An analogue of Remark \ref{rem_B1} holds, in particular, $B_{3}$ is asymptotically sharp for \eqref{GN_hyper1}.
\end{rem}
\begin{rem} In \cite{INW14}, similar inequalities to \eqref{analog_LT13_1} and \eqref{GN_hyper1} are investigated for radially symmetric functions in $\Rn$.
\end{rem}
\begin{proof}[Proof of Theorem \ref{GN_hyper_thm}] Since $B_{3}\leq A_{3}$, then, as in the proof of Theorem \ref{Hardy_hyper_thm}, we show the following two cases: \eqref{analog_LT13_1}$\Rightarrow$\eqref{GN_hyper1} with $\alpha_{\beta_{1}}\leq(en'A_{3}^{n'})^{-1}$ and
\eqref{GN_hyper1}$\Rightarrow$\eqref{analog_LT13_1} with $1/\alpha_{\beta_{1}}\leq n'eB_{3}^{n'}$. So, we start to show \eqref{analog_LT13_1}$\Rightarrow$\eqref{GN_hyper1} with $\alpha_{\beta_{1}}\leq(en'A_{3}^{n'})^{-1}$. In the case $\|\nabla_{g} f\|_{L^{n}(\mathbb{H}^{n})}=0$ we have $f\equiv0$ by taking into account the definition of $f\in W_{0}^{1,n}(\mathbb{H}^{n})$, so there is nothing to prove. Therefore, we can assume that $\|\nabla_{g} f\|_{L^{n}(\mathbb{H}^{n})}\neq0$. Replacing $f$ by $f/\|\nabla_{g} f\|_{L^{n}(\mathbb{H}^{n})}$ in \eqref{analog_LT13_1} we get
\begin{equation}\label{GN_hyper3_0}
\int_{\mathbb{H}^{n}}\frac{1}{\rho^{\beta_{1}}}\sum_{k=n-1}^{\infty}
\frac{\alpha^{k}|f(x)|^{kn'}}{k!}dV_{g}\leq \widetilde{C_{3}}
\|\nabla_{g} f\|_{L^{n}(\mathbb{H}^{n})}^{kn'-n(1-\delta)}\left(\int_{\mathbb{H}^{n}}
\frac{|f(x)|^{n}}{\rho^{\beta_{2}}}dV_{g}\right)^{1-\delta}.
\end{equation}
It implies that for any $\varepsilon$ with $0<\varepsilon<\alpha_{\beta_{1}}$ there exists $C_{\varepsilon}$ such that
\begin{equation}\label{GN_hyper3}
\int_{\mathbb{H}^{n}}\frac{1}{\rho^{\beta_{1}}}\sum_{k=n-1}^{\infty}
\frac{(\alpha_{\beta_{1}}-\varepsilon)^{k}|f(x)|^{kn'}}{k!}dV_{g}\leq C_{\varepsilon}
\|\nabla_{g} f\|_{L^{n}(\mathbb{H}^{n})}^{kn'-n(1-\delta)}\left(\int_{\mathbb{H}^{n}}
\frac{|f(x)|^{n}}{\rho^{\beta_{2}}}dV_{g}\right)^{1-\delta}.
\end{equation}
In particular, we get from this that
\begin{equation}\label{GN_hyper4}
\left\|\frac{f}{\rho^{\frac{\beta_{1}}{kn'}}}\right\|_{L^{kn'}(\mathbb{H}^{n})}\leq (C_{\varepsilon}k!)^{1/(kn')}
(\alpha_{\beta_{1}}-\varepsilon)^{-1/n'}\|\nabla_{g}
f\|_{L^{n}(\mathbb{H}^{n})}^{1-\frac{n(1-\delta)}{n'k}}
\left\|\frac{f}{\rho^{\frac{\beta_{2}}{n}}}\right\|_{L^{n}(\mathbb{H}^{n})}^{\frac{n(1-\delta)}{n'k}}
\end{equation}
for all $k\geq n-1$. Moreover, for any $q\geq n$, there exists an integer $k\geq n-1$ satisfying $n'k\leq q <n'(k+1)$. Then, combining \eqref{Hardy_hyper5} with \eqref{GN_hyper4}, one gets
\begin{equation}\label{GN_hyper6}
\left\|\frac{f}{\rho^{\frac{\beta_{1}}{q}}}\right\|_{L^{q}(\mathbb{H}^{n})}\leq
C_{\varepsilon}^{\frac{1}{q}}(\alpha_{\beta_{1}}-\varepsilon)^{-\frac{1}{n'}}
((k+1)!)^{\frac{1}{q}}\|\nabla_{g} f\|_{L^{n}(\mathbb{H}^{n})}^{1-\frac{n(1-\delta)}{q}}
\left\|\frac{f}{\rho^{\frac{\beta_{2}}{n}}}\right\|_{L^{n}(\mathbb{H}^{n})}^{\frac{n(1-\delta)}{q}}.
\end{equation}
Since $(k+1)!\leq \Gamma(q/n'+2)$ for $q\geq n'k$, we rewrite \eqref{GN_hyper6} as
\begin{equation}\label{GN_hyper7}
\left\|\frac{f}{\rho^{\frac{\beta_{1}}{q}}
}\right\|_{L^{q}(\mathbb{H}^{n})}\leq
(C_{\varepsilon}\Gamma(q/n'+2))^{1/q}(\alpha_{\beta_{1}}-\varepsilon)^{-1/n'}\|\nabla_{g} f\|_{L^{n}(\mathbb{H}^{n})}^{1-\frac{n(1-\delta)}{q}}
\left\|\frac{f}{\rho^{\frac{\beta_{2}}{n}}}\right\|_{L^{n}(\mathbb{H}^{n})}^{\frac{n(1-\delta)}{q}}
\end{equation}
for any $q\geq n$ and for all $f\in W_{0}^{1,n}(\mathbb{H}^{n})$, which is \eqref{GN_hyper1}. With the help of \eqref{Gamma1} for $q\rightarrow+\infty$ and \eqref{GN_hyper7}, we know that for any $\delta>0$ there exists $r\geq n$ such that
\begin{equation}\label{GN_hyper8}
\left\|\frac{f}{\rho^{\frac{\beta_{1}}{q}}
}\right\|_{L^{q}(\mathbb{H}^{n})}\leq
((n'e(\alpha_{\beta_{1}}-\varepsilon))^{-1/n'}+\delta)q^{1-1/n}\|\nabla_{g} f\|_{L^{n}(\mathbb{H}^{n})}^{1-\frac{n(1-\delta)}{q}}
\left\|\frac{f}{\rho^{\frac{\beta_{2}}{n}}}\right\|_{L^{n}(\mathbb{H}^{n})}^{\frac{n(1-\delta)}{q}}
\end{equation}
holds for all $f\in W_{0}^{1,n}(\mathbb{H}^{n})$ and all $q$ with $r\leq q<\infty$.

Thus, we note that $A_{3}\leq (n'e(\alpha_{\beta_{1}}-\varepsilon))^{-1/n'}+\delta$, then by the arbitrariness of $\varepsilon$ and
$\delta$ we obtain $\alpha_{\beta_{1}}\leq(en'A_{3}^{n'})^{-1}$.

Now let us show that \eqref{GN_hyper1}$\Rightarrow$\eqref{analog_LT13_1} with $1/\alpha_{\beta_{1}}\leq n'eB_{3}^{n'}$. By
\eqref{GN_hyper1}, for any $q$ with $n\leq q<\infty$ there exists $C_{3}=C_{3}(n,\beta_{1},\beta_{2},q)>0$ such that
\begin{equation}\label{GN_hyper9}
\left\|\frac{f}{\rho^{\frac{\beta_{1}}{q}}}\right\|_{L^{q}(\mathbb{H}^{n})}\leq C_{3}q^{1-1/n}\|\nabla_{g}
f\|_{L^{n}(\mathbb{H}^{n})}^{1-\frac{n(1-\delta)}{q}}
\left\|\frac{f}{\rho^{\frac{\beta_{2}}{n}}}\right\|_{L^{n}(\mathbb{H}^{n})}^{\frac{n(1-\delta)}{q}}
\end{equation}
holds for all $f\in W_{0}^{1,n}(\mathbb{H}^{n})$.
Using this and $\|\nabla_{g} f\|_{L^{n}(\mathbb{H}^{n})}\leq1$, we arrive at
\begin{multline}\label{GN_hyper10}
\int_{\mathbb{H}^{n}}\frac{1}{\rho^{\beta_{1}}}\left(\exp(\alpha|f(x)|^{n'})-\sum_{k=0}^{n-2}
\frac{(\alpha|f(x)|^{n'})^{k}}{k!}\right)dV_{g}\\ \leq
\sum_{n'k\geq n,\;k\in\mathbb{N}}\frac{(\alpha
n'kC_{3}^{n'})^{k}}{k!}\left(\int_{\mathbb{H}^{n}}
\frac{|f(x)|^{n}}{\rho^{\beta_{2}}}dV_{g}\right)^{1-\delta}.
\end{multline}
The series in the right hand side of \eqref{GN_hyper10} converges when $0\leq\alpha<1/(n'eC_{3}^{n'})$. Thus, we have obtained \eqref{analog_LT13_1} with $0\leq\alpha<1/(n'eC_{3}^{n'})$. Hence
$\alpha_{\beta_{1}}\geq 1/(n'eC_{3}^{n'})$ for all $C_{3}\geq B_{3}$, which gives $\alpha_{\beta_{1}}\geq 1/(n'eB_{3})^{n'}$.

Thus, we have completed the proof of Theorem \ref{GN_hyper_thm}.
\end{proof}
We now recall another version of the weighted Trudinger-Moser inequality with a more explicit expression for the exponent for radially decreasing functions.

\begin{thm}[{\cite[Theorem 1]{LC17}}]
\label{Trudinger_GN2_hyper_thm}
Let $\mathbb{H}^{n}\;(n\geq2)$ be the $n$-dimensional hyperbolic space. Let $0\leq\beta<n$ and let $0<\alpha\leq \alpha_{\beta}$ with $\alpha_{\beta}=n\omega_{n-1}^{1/(n-1)}(1-\beta/n)$. Then there exists a positive constant $\widetilde{C_{4}}=\widetilde{C_{4}}(\beta, n, \alpha)$ such that
\begin{multline}\label{Trudinger_GN2_hyper1}
\int_{\mathbb{H}^{n}}\frac{1}{(1+|f(x)|)^{n/(n-1)}\rho^{\beta}}\left(\exp(\alpha|f(x)|^{n/(n-1)})\right.\\\left.-\sum_{k=0}^{n-2}\frac{
\alpha^{k}|f(x)|^{kn/(n-1)}}{k!}\right)dV_{g}\leq \widetilde{C_{4}}\int_{\mathbb{H}^{n}}
\frac{|f(x)|^{n}}{\rho^{\beta}}dV_{g}
\end{multline}
holds for all radially decreasing functions $f\in
W_{0}^{1,n}(\mathbb{H}^{n})$ with $\|\nabla_{g} f\|_{L^{n}(\mathbb{H}^{n})}\leq1$, where $\omega_{n-1}$ is the area of the surface of the unit $n$-ball in $\mathbb{H}^{n}$. Moreover, the power $n/(n-1)$ in the denominator is sharp.
\end{thm}
\begin{rem}\label{rem_LT16} In \cite[Theorem 1.1]{LT16}, the authors proved that the constant $\alpha_{\beta}$ in \eqref{Trudinger_GN2_hyper1} is sharp when $\beta=0$ for all functions, not necessarily being radially decreasing.
\end{rem}
Let us show that actually Theorem \ref{Trudinger_GN2_hyper_thm} holds for any function $f\in W_{0}^{1,n}(\mathbb{H}^{n})$ dropping the radial assumption.
\begin{thm}\label{Trudinger_GN2_hyper_analog_thm}
Let $\mathbb{H}^{n}\;(n\geq2)$ be the $n$-dimensional hyperbolic space. Let $0\leq\beta<n$ and let $0<\alpha<\alpha_{\beta}$ with $\alpha_{\beta}=n\omega_{n-1}^{1/(n-1)}(1-\beta/n)$. Then there exists a positive constant $\widetilde{C_{5}}=\widetilde{C_{5}}(\beta, n, \alpha)$ such that
\begin{multline}\label{Trudinger_GN2_hyper1_analog}
\int_{\mathbb{H}^{n}}\frac{1}{(1+|f(x)|)^{n/(n-1)}\rho^{\beta}}\left(\exp(\alpha|f(x)|^{n/(n-1)})\right.\\\left.-\sum_{k=0}^{n-2}\frac{
\alpha^{k}|f(x)|^{kn/(n-1)}}{k!}\right)dV_{g}\leq \widetilde{C_{5}}\int_{\mathbb{H}^{n}}
\frac{|f(x)|^{n}}{\rho^{\beta}}dV_{g}
\end{multline}
holds for all functions $f\in
W_{0}^{1,n}(\mathbb{H}^{n})$ with $\|\nabla_{g} f\|_{L^{n}(\mathbb{H}^{n})}\leq1$, where $\omega_{n-1}$ is the area of the surface of the unit $n$-ball in $\mathbb{H}^{n}$. Moreover, the power $n/(n-1)$ in the denominator is sharp.
\end{thm}
\begin{proof}[Proof of Theorem \ref{Trudinger_GN2_hyper_analog_thm}] By Theorem \ref{GN_hyper_thm} with $\beta_{1}=\beta_{2}=\beta$, hence $\delta=0$ by \eqref{delta}, we obtain
\begin{multline*}
\left\|\frac{f}{\rho^{\frac{\beta}{q}}(1+|f|)^{\frac{n'}{q}}
}\right\|_{L^{q}(\mathbb{H}^{n})}\leq
\left\|\frac{f}{\rho^{\frac{\beta}{q}}}\right\|_{L^{q}(\mathbb{H}^{n})}\leq B_{3}q^{1-1/n}\|\nabla_{g}
f\|_{L^{n}(\mathbb{H}^{n})}^{1-n/q}
\left\|\frac{f}{\rho^{\frac{\beta}{n}}}\right\|_{L^{n}(\mathbb{H}^{n})}^{n/q},
\end{multline*}
where $B_{3}$ is given in Theorem \ref{GN_hyper_thm}.
Then, using this and $\|\nabla_{g} f\|_{L^{n}(\mathbb{H}^{n})}\leq1$, one gets
\begin{multline}\label{GN2_hyper10_trud}
\int_{\mathbb{H}^{n}}\frac{1}{\rho^{\beta}(1+|f(x)|)^{n'}}\left(\exp(\alpha|f(x)|^{n'})-\sum_{k=0}^{n-2}
\frac{(\alpha|f(x)|^{n'})^{k}}{k!}\right)dV_{g}\\ \leq
\sum_{n'k\geq n,\;k\in\mathbb{N}}\frac{(\alpha
n'kB_{3}^{n'})^{k}}{k!}\int_{\mathbb{H}^{n}}
\frac{|f(x)|^{n}}{\rho^{\beta}}dV_{g}.
\end{multline}
The series in the right hand side of \eqref{GN2_hyper10_trud} converges when $0\leq\alpha<1/(n'eB_{3}^{n'})$. Thus, we have obtained \eqref{Trudinger_GN2_hyper1_analog} with $0\leq\alpha<1/(n'eB_{3}^{n'})$. Since $B_{3}=(\alpha_{\beta}n'e)^{-1/n'}$ by Remark \ref{rem_B3}, then we have \eqref{Trudinger_GN2_hyper1_analog} for $0\leq\alpha<\alpha_{\beta}$.

The sharpness of the power $n/(n-1)$ in the denominator is obtained by Theorem \ref{Trudinger_GN2_hyper_thm}, since this constant is sharp for radially decreasing functions in \eqref{Trudinger_GN2_hyper1_analog}.
\end{proof}
Now we show that \eqref{Trudinger_GN2_hyper1_analog} is equivalent to the following Caffarelli-Kohn-Nirenberg type inequalities.
\begin{thm}\label{GN3_hyper_thm}
Let $\mathbb{H}^{n}\;(n\geq2)$ be the $n$-dimensional hyperbolic space and let $0\leq\beta<n$. Then for any $n\leq q<\infty$ there exists a positive constant $C_{5}=C_{5}(n,\beta,q)$ such that
\begin{equation}\label{GN3_hyper}
\left\|\frac{f}{\rho^{\frac{\beta}{q}}(1+|f|)^{\frac{n'}{q}}
}\right\|_{L^{q}(\mathbb{H}^{n})}\leq C_{5} q^{1-1/n}\|\nabla_{g} f\|_{L^{n}(\mathbb{H}^{n})}^{1-n/q}
\left\|\frac{f}{\rho^{\frac{\beta}{n}}}\right\|_{L^{n}(\mathbb{H}^{n})}^{n/q}
\end{equation}
holds for all functions $f\in W_{0}^{1,n}(\mathbb{H}^{n})$. Moreover, we have
\begin{equation}\label{equiv_identity_GN3_hyper}
\frac{1}{\widetilde{\alpha}_{\beta} n'e}=A_{5}^{n'}=B_{5}^{n'},
\end{equation}
where
\begin{multline*}
\widetilde{\alpha}_{\beta}=\sup\{\alpha>0; \exists \widetilde{C_{5}}=\widetilde{C_{5}}(\beta, n, \alpha):\eqref{Trudinger_GN2_hyper1_analog}\\ \textrm{ holds for all functions} \\ f\in
W_{0}^{1,n}(\mathbb{H}^{n})
\textrm{ with } \|\nabla_{g} f\|_{L^{n}(\mathbb{H}^{n})}\leq1\},
\end{multline*}
\begin{equation*}
\begin{split}
A_{5}=\inf\{C_{5}>0; \exists r=r(n,&\beta,C_{5}) \textrm{ with } r\geq n:\\& \eqref{GN3_hyper}\textrm{ holds }\forall f\in
W_{0}^{1,n}(\mathbb{H}^{n}),
\forall q
\textrm{ with } r\leq q<\infty\},
\end{split}
\end{equation*}
\begin{equation}\label{alphaDF_GN3}
B_{5}=\limsup_{q\rightarrow \infty}\sup_{f\in W^{1,n}(\mathbb{H}^{n})\backslash\{0\}}
\frac{\left\|\frac{f}{\rho^{\frac{\beta}{q}}}\right\|_{L^{q}(\mathbb{H}^{n})}}{q^{1-1/n}\|\nabla_{g}
f\|_{L^{n}(\mathbb{H}^{n})}^{1-n/q}
\left\|\frac{f}{\rho^{\frac{\beta}{n}}}\right\|_{L^{n}(\mathbb{H}^{n})}^{n/q}}.
\end{equation}
The weighted Trudinger-Moser inequalities \eqref{Trudinger_GN2_hyper1_analog} are equivalent to the Caffarelli-Kohn-Nirenberg type inequalities \eqref{GN3_hyper} with relation \eqref{equiv_identity_GN3_hyper}.
\end{thm}
\begin{rem}\label{rem_LT16_2} An analogue of Remark \ref{rem_B1} holds, in particular, $B_{5}$ is asymptotically sharp for \eqref{GN3_hyper}. When $\beta=0$, by Remark \ref{rem_LT16} we have $\widetilde{\alpha}_{\beta}=\alpha_{\beta}=n\omega_{n-1}^{1/(n-1)}$, which gives explicit expression for $B_{5}=(n\omega_{n-1}^{1/(n-1)}n'e)^{-1/n'}$.
\end{rem}
\begin{proof}[Proof of Theorem \ref{GN3_hyper_thm}] Since $B_{5}\leq A_{5}$, then, we need to show the following two cases: \eqref{Trudinger_GN2_hyper1_analog}$\Rightarrow$\eqref{GN3_hyper} with $\widetilde{\alpha}_{\beta}\leq(en'A_{5}^{n'})^{-1}$ and
\eqref{GN3_hyper}$\Rightarrow$\eqref{Trudinger_GN2_hyper1_analog} with $1/\widetilde{\alpha}_{\beta}\leq n'eB_{5}^{n'}$. So, we start to show \eqref{Trudinger_GN2_hyper1_analog}$\Rightarrow$\eqref{GN3_hyper} with $\widetilde{\alpha}_{\beta}\leq(en'A_{5}^{n'})^{-1}$. As in the proof of Theorem \ref{GN_hyper_thm}, it suffices to show \eqref{GN3_hyper} for $\|\nabla_{g} f\|_{L^{n}(\mathbb{H}^{n})}\neq0$. Then, we replace $f$ by $f/\|\nabla_{g} f\|_{L^{n}(\mathbb{H}^{n})}$ in
\eqref{Trudinger_GN2_hyper1_analog} with $0<\alpha<\widetilde{\alpha}_{\beta}$ to get
\begin{equation}\label{GN2_hyper3_0}
\begin{split}
\int_{\mathbb{H}^{n}}\frac{1}{\rho^{\beta}\left(1+\frac{|f(x)|}{\|\nabla_{g}
f\|_{L^{n}(\mathbb{H}^{n})}}\right)^{n'}}\sum_{k=n-1}^{\infty}&
\frac{\alpha^{k}|f(x)|^{kn'}}{k!}dV_{g}\\&\leq \widetilde{C_{5}}
\|\nabla_{g} f\|_{L^{n}(\mathbb{H}^{n})}^{kn'-n}\left(\int_{\mathbb{H}^{n}}
\frac{|f(x)|^{n}}{\rho^{\beta}}dV_{g}\right).
\end{split}
\end{equation}
From this, we note that for any $0<\varepsilon<\widetilde{\alpha}_{\beta}$ there is $C_{\varepsilon}$ such that
\begin{equation}\label{GN2_hyper3}
\begin{split}
\int_{\mathbb{H}^{n}}\frac{1}{\rho^{\beta}\left(1+\frac{|f(x)|}{\|\nabla_{g}
f\|_{L^{n}(\mathbb{H}^{n})}}\right)^{n'}}\sum_{k=n-1}^{\infty}&
\frac{(\widetilde{\alpha}_{\beta}-\varepsilon)^{k}|f(x)|^{kn'}}{k!}dV_{g}\\&\leq C_{\varepsilon}
\|\nabla_{g} f\|_{L^{n}(\mathbb{H}^{n})}^{kn'-n}\left(\int_{\mathbb{H}^{n}}
\frac{|f(x)|^{n}}{\rho^{\beta}}dV_{g}\right).
\end{split}
\end{equation}
In particular, it implies that
\begin{multline}\label{GN2_hyper4}
\left\|\frac{f}{\rho^{\frac{\beta}{kn'}}\left(1+\frac{|f|}{\|\nabla_{g}
f\|_{L^{n}(\mathbb{H}^{n})}}\right)^{1/k}}\right\|_{L^{kn'}(\mathbb{H}^{n})}\\ \leq (C_{\varepsilon}k!)^{1/(kn')}
(\widetilde{\alpha}_{\beta}-\varepsilon)^{-1/n'}\|\nabla_{g}
f\|_{L^{n}(\mathbb{H}^{n})}^{1-(n-1)/k}\left\|\frac{f}{\rho^{\frac{\beta}{n}}}\right\|_{L^{n}(\mathbb{H}^{n})}^{(n-1)/k}
\end{multline}
holds for all $k\geq n-1$. Moreover, for any $q\geq n$, there exists an integer $k\geq n-1$ satisfying $n'k\leq q <n'(k+1)$. Then, using H\"{o}lder's inequality for $\frac{\theta q}{n'k}+\frac{(1-\theta)
q}{n'(k+1)}=1$ with $0<\theta\leq1$ we calculate
\begin{equation*}
\begin{split}
\int_{\mathbb{H}^{n}}&\frac{|f(x)|^{q}}{\rho^{\beta}\left(1+\frac{|f(x)|}{\|\nabla_{g} f\|_{L^{n}(\mathbb{H}^{n})}}\right)^{n'}}dV_{g}\\&
=\int_{\mathbb{H}^{n}}\frac{|f(x)|^{\theta q}}{\rho^{\frac{\beta\theta
q}{n'k}}\left(1+\frac{|f(x)|}{\|\nabla_{g} f\|_{L^{n}(\mathbb{H}^{n})}}\right)^{\frac{\theta
q}{k}}}\cdot\frac{|f(x)|^{(1-\theta)q}}{\rho^{\frac{\beta(1-\theta) q}{n'(k+1)}}\left(1+\frac{|f(x)|}{\|\nabla_{g}
f\|_{L^{n}(\mathbb{H}^{n})}}\right)^{\frac{(1-\theta) q}{k+1}}}dV_{g}\\&
\leq \left(\int_{\mathbb{H}^{n}}\frac{|f(x)|^{n'k}}{\rho^{\beta}\left(1+\frac{|f(x)|}{\|\nabla_{g}
f\|_{L^{n}(\mathbb{H}^{n})}}\right)^{n'}}dV_{g}\right)^{\frac{\theta q}{n'k}}
\left(\int_{\mathbb{H}^{n}}\frac{|f(x)|^{n'(k+1)}}{\rho^{\beta}\left(1+\frac{|f(x)|}{\|\nabla_{g}
f\|_{L^{n}(\mathbb{H}^{n})}}\right)^{n'}}dV_{g}\right)^{\frac{(1-\theta) q}{n'(k+1)}}\\&
=\left\|\frac{f}{\rho^{\frac{\beta}{n'k}}\left(1+\frac{|f|}{\|\nabla_{g}
f\|_{L^{n}(\mathbb{H}^{n})}}\right)^{1/k}}\right\|_{L^{n'k}(\mathbb{H}^{n})}^{\theta q}
\left\|\frac{f}{\rho^{\frac{\beta}{n'(k+1)}}\left(1+\frac{|f|}{\|\nabla_{g}
f\|_{L^{n}(\mathbb{H}^{n})}}\right)^{1/(k+1)}}\right\|_{L^{n'(k+1)}(\mathbb{H}^{n})}^{(1-\theta)q},
\end{split}
\end{equation*}
that is,
\begin{equation*}
\begin{split}
&\left\|\frac{f}{\rho^{\frac{\beta}{q}}\left(1+\frac{|f|}{\|\nabla_{g} f\|_{L^{n}(\mathbb{H}^{n})}}\right)^{\frac{n'}{q}}
}\right\|_{L^{q}(\mathbb{H}^{n})}
\\& \leq\left\|\frac{f}{\rho^{\frac{\beta}{n'k}}\left(1+\frac{|f|}{\|\nabla_{g}
f\|_{L^{n}(\mathbb{H}^{n})}}\right)^{1/k}}\right\|_{L^{n'k}(\mathbb{H}^{n})}^{\theta}
\left\|\frac{f}{\rho^{\frac{\beta}{n'(k+1)}}\left(1+\frac{|f|}{\|\nabla_{g}
f\|_{L^{n}(\mathbb{H}^{n})}}\right)^{1/(k+1)}}\right\|_{L^{n'(k+1)}(\mathbb{H}^{n})}^{1-\theta}.
\end{split}
\end{equation*}
Combining this with \eqref{GN2_hyper4}, we obtain
\begin{multline}\label{GN2_hyper6}
\left\|\frac{f}{\rho^{\frac{\beta}{q}}\left(1+\frac{|f|}{\|\nabla_{g} f\|_{L^{n}(\mathbb{H}^{n})}}\right)^{\frac{n'}{q}}
}\right\|_{L^{q}(\mathbb{H}^{n})}\\ \leq
C_{\varepsilon}^{\frac{1}{q}}(\widetilde{\alpha}_{\beta}-\varepsilon)^{-\frac{1}{n'}}
((k+1)!)^{\frac{1}{q}}\|\nabla_{g} f\|_{L^{n}(\mathbb{H}^{n})}^{1-n/q}
\left\|\frac{f}{\rho^{\frac{\beta}{n}}}\right\|_{L^{n}(\mathbb{H}^{n})}^{n/q}.
\end{multline}
Since $q\geq n'k$ we get $(k+1)!\leq \Gamma(q/n'+2)$, then \eqref{GN2_hyper6} gives that
\begin{multline}\label{GN2_hyper7}
\left\|\frac{f}{\rho^{\frac{\beta}{q}}\left(1+\frac{|f|}{\|\nabla_{g} f\|_{L^{n}(\mathbb{H}^{n})}}\right)^{\frac{n'}{q}}
}\right\|_{L^{q}(\mathbb{H}^{n})}\\ \leq
C_{\varepsilon}^{\frac{1}{q}}(\widetilde{\alpha}_{\beta}-\varepsilon)^{-\frac{1}{n'}}
(\Gamma(q/n'+2))^{\frac{1}{q}}\|\nabla_{g}f\|_{L^{n}(\mathbb{H}^{n})}^{1-n/q}
\left\|\frac{f}{\rho^{\frac{\beta}{n}}}\right\|_{L^{n}(\mathbb{H}^{n})}^{n/q}
\end{multline}
for any $q\geq n$ and for all $f\in W_{0}^{1,n}(\mathbb{H}^{n})$, which gives \eqref{GN3_hyper} after replacing $f$ by $\|\nabla_{g} f\|_{L^{n}(\mathbb{H}^{n})}f$. For $q\rightarrow+\infty$, using \eqref{Gamma1} in \eqref{GN2_hyper7} we see that for any $\delta>0$ there is $r\geq n$ such that
\begin{multline}\label{GN2_hyper8}
\left\|\frac{f}{\rho^{\frac{\beta}{q}}\left(1+\frac{|f|}{\|\nabla_{g} f\|_{L^{n}(\mathbb{H}^{n})}}\right)^{\frac{n'}{q}}
}\right\|_{L^{q}(\mathbb{H}^{n})}\\ \leq
((n'e(\widetilde{\alpha}_{\beta}-\varepsilon))^{-\frac{1}{n'}}+\delta)
q^{\frac{1}{n'}}\|\nabla f\|_{L^{n}(\mathbb{H}^{n})}^{1-n/q}
\left\|\frac{f}{\rho^{\frac{\beta}{n}}}\right\|_{L^{n}(\mathbb{H}^{n})}^{n/q}
\end{multline}
holds for all $f\in W_{0}^{1,n}(\mathbb{H}^{n})$ and all $q$ with $r\leq q<\infty$. Here, replacing $f$ by $\|\nabla_{g} f\|_{L^{n}(\mathbb{H}^{n})}f$ we obtain \eqref{GN3_hyper}. We see that $A_{5}\leq (n'e(\widetilde{\alpha}_{\beta}-\varepsilon))^{-1/n'}+\delta$, then by the arbitrariness of $\varepsilon$ and $\delta$ we obtain $\widetilde{\alpha}_{\beta}\leq(en'A_{5}^{n'})^{-1}$.

It remains to show that \eqref{GN3_hyper}$\Rightarrow$\eqref{Trudinger_GN2_hyper1_analog} with $1/\widetilde{\alpha}_{\beta}\leq n'eB_{5}^{n'}$. By
\eqref{GN3_hyper}, for any $q$ with $n\leq q<\infty$ there exists $C_{5}=C_{5}(n,\beta,q)>0$ such that
\begin{equation}\label{GN2_hyper9}
\left\|\frac{f}{\rho^{\frac{\beta}{q}}(1+|f|)^{\frac{n'}{q}}
}\right\|_{L^{q}(\mathbb{H}^{n})}\leq
C_{5}q^{1-1/n}\|\nabla_{g}f\|_{L^{n}(\mathbb{H}^{n})}^{1-n/q}
\left\|\frac{f}{\rho^{\frac{\beta}{n}}}\right\|_{L^{n}(\mathbb{H}^{n})}^{n/q}
\end{equation}
holds for all $f\in W_{0}^{1,n}(\mathbb{H}^{n})$.
By this and taking into account $\|\nabla_{g} f\|_{L^{n}(\mathbb{H}^{n})}\leq1$, one gets
\begin{multline}\label{GN2_hyper10}
\int_{\mathbb{H}^{n}}\frac{1}{\rho^{\beta}(1+|f(x)|)^{n'}}\left(\exp(\alpha|f(x)|^{n'})-\sum_{k=0}^{n-2}
\frac{(\alpha|f(x)|^{n'})^{k}}{k!}\right)dV_{g}\\ \leq
\sum_{n'k\geq n,\;k\in\mathbb{N}}\frac{(\alpha
n'kC_{5}^{n'})^{k}}{k!}\int_{\mathbb{H}^{n}}
\frac{|f(x)|^{n}}{\rho^{\beta}}dV_{g}.
\end{multline}
The series in the right hand side of \eqref{GN2_hyper10} converges when $0\leq\alpha<1/(n'eC_{5}^{n'})$. Thus, we have obtained \eqref{Trudinger_GN2_hyper1_analog} with $0\leq\alpha<1/(n'eC_{5}^{n'})$. Hence
$\widetilde{\alpha}_{\beta}\geq 1/(n'eC_{5}^{n'})$ for all $C_{5}\geq B_{5}$, which gives $\widetilde{\alpha}_{\beta}\geq 1/(n'eB_{5})^{n'}$.

Thus, we have completed the proof of Theorem \ref{GN3_hyper_thm}.
\end{proof}

\end{document}